\documentclass[11pt, reqno]{amsart}
\usepackage[a4paper,vmargin={3cm,3cm},hmargin={2.5cm,2.5cm},bottom=30mm]{geometry}
\usepackage{amsfonts}
\usepackage{amsthm,amsmath}
\usepackage{amsbsy}
\usepackage{bbm}
\usepackage{natbib}
\usepackage{amssymb} 
\usepackage{verbatim}
 \usepackage{color}
\usepackage{csquotes}

\usepackage[mathscr]{eucal}
\usepackage{times}
\usepackage[pdfpagemode=UseNone,bookmarksopen=false,colorlinks=blue]{hyperref}
\hypersetup{ colorlinks=true, linkcolor=blue, citecolor=blue,
  filecolor=blue, urlcolor=blue}

\newlength{\myfboxsep}
\newlength{\mywidth}
\newcommand{\boxv}{\setlength{\myfboxsep}{\fboxsep}
\setlength{\fboxsep}{0pt}
\,\framebox[\mywidth]{$\vee$}\,
\setlength{\fboxsep}{\myfboxsep}}

\allowdisplaybreaks
\newtheorem{theorem}{Theorem}[section]
\newtheorem{lemma}{Lemma}[section]
\newtheorem{proposition}{Proposition}[section]

\newtheorem{remark}{Remark}[section]

\theoremstyle{definition}
\newtheorem{definition}{Definition}[section]

\newtheorem{question}{Question}[section]
\theoremstyle{remark}

\numberwithin{equation}{section}

\def\maxbool{\mathrel{%
    \mathchoice{\QEQ}{\QEQ}{\scriptsize\QEQ}{\tiny\QEQ}%
}}

\def\QEQ{{%
    \setbox0\hbox{$\cup$}%
    \rlap{\hbox to \wd0{\hss $\text{ }\vee$ \hss}}\box0
}}

\def\minbool{\mathrel{%
    \mathchoice{\REQ}{\REQ}{\scriptsize\REQ}{\tiny\REQ}%
}}

\def\REQ{{%
\setbox0\hbox{$\cup$}%
    \rlap{\hbox to \wd0{\hss $\text{ }\wedge$ \hss}}\box0
}}

\DeclareMathOperator{\lito}{o}
\DeclareMathOperator{\bigo}{O}

\newcommand{\mboli}{\!\begin{array}{c} {\scriptstyle\times} \\[-12pt]\cup\end{array}\!}

\makeatletter
\def\blfootnote{\gdef\@thefnmark{}\@footnotetext}
\makeatother

\begin{document}
\title{Boolean convolutions and regular variation}
\author{Sukrit Chakraborty and Rajat Subhra Hazra}
\address{ Indian Statistical Institute\\ 203, B.T. Road, Kolkata- 700108, India }
\email{sukrit049@gmail.com, rajatmaths@gmail.com}
\blfootnote{\textup{2010} \textit{Mathematics Subject Classification}: 60G70; 46L53; 46L54.}
\keywords{Boolean convolution, regular variation, free convolution, subexponential}

\newcommand{\acr}{\newline\indent}
\begin{abstract}
In this article we study the influence of regularly varying probability measures on additive and multiplicative Boolean convolutions.  We introduce the notion of Boolean subexponentiality (for additive Boolean convolution), which extends the notion of classical and free subexponentiality. We show that the distributions with regularly varying tails belong to the class of Boolean subexponential distributions. As an application we also study the behaviour of the Belinschi-Nica map. Breiman's theorem studies the classical product convolution between regularly varying measures. We derive an analogous result to Breiman's theorem in case of multiplicative Boolean convolution.  In proving these results we exploit the relationship of regular variation with different transforms and their Taylor series expansion.
\end{abstract}
\maketitle

\section{Introduction}
In the set up of quantum probability the notion of stochastic independence plays a very crucial role. M. Schurmann (\cite{schurmann}) conjectured that there are only three notions of independence arising out of algebraic probability spaces. This conjecture was established by \cite{speicher:1997} who showed that tensor (classical), free and Boolean independence are the only ones. See for further details~\cite{booliremark1, muraki}. The main aim of this article is to study the Boolean convolution and its properties when the measures belong to the class of heavy tailed random variables.

The additive Boolean convolution of two probability measures $\mu$ and $\nu$ on the real line (denoted by $\mu \uplus \nu$) was introduced in ~\cite{rezspeibooli} and the multiplicative Boolean convolution of two probability measures $\mu$ and $\nu$ (denoted by $\mu \mboli \nu$) was introduced in ~\cite{bercovicibooli} where $\mu$ and $\nu$ are both defined on the non-negative part of the real line. Later  Franz introduced the concept of Boolean independence and  defined Boolean convolutions using operator theory in \cite{franzboolean}, which is similar to the approach of Bercovici and Voiculescu for the free convolutions in ~\cite{bercovici1993free}. The definitions of Boolean convolutions using Boolean independence also agree with the former definitions. 
 
 In this article we are interested in a certain class of measures having power law tail behaviour. A measure is called regularly varying of index $-\alpha$, for some $\alpha>0$, if $\mu(x,\infty)\sim x^{-\alpha} L(x)$ for some slowly varying function $L(x)$ (see explicit definition in next section). Such measures form  a large class containing important distributions like Pareto and Fr\'echet and classical stable laws. The class of distribution functions with regularly varying tail index $-\alpha$, $\alpha  \geqslant 0$ delivers significant applications in finance, insurance, weather, Internet traffic modelling and many other fields. In this paper, we want to realise what happens with the Boolean convolutions of the probability measures which have regularly varying tails. In particular we want to address the following question:
\begin{question}
Suppose $\mu$ and $\nu$ are two probability measures supported on $[0,\infty)$ with regularly varying tails of indices $-\alpha$ and $-\beta$ respectively ($\alpha$ and $\beta$ non-negative). Then what can be said about the tail behaviour of $\mu \uplus \nu$ and $\mu \mboli \nu$?
\end{question}

When one considers the case of classical additive convolution, the answer is well known and the principle of one large jump gives that the heavier tail dominates. In fact it is well known, if $\mu$ is regularly varying of index $-\alpha$, then $\mu$ is (classical) subexponential in the sense that $\mu^{\ast n}(x,\infty)\sim n \mu(x,\infty)$ as $x\to \infty$ for all $n\geqslant  1$. For a contemporary review on subexponential distributions and their applications we refer to~\cite{foss:korshu:stan, goldie1998, jessenmikosch}. The case of free additive convolution was studied by~\cite{hazra1} and it is related to the free extreme value theory of~\cite{extremeben}. One of the main aim of this article is to  extend this result to the Boolean additive convolution.

The case of multiplicative convolution turns out to be more interesting and challenging. In classical independence, the role of Breiman's theorem is very crucial in this result. In an influential work of Breiman (\cite{breiman:1965}) he showed the following: If $\mu$ and $\nu$ are positively supported measures, $\mu$ is regularly varying of index $-\alpha$ and $\nu$ is such that $\int_0^\infty y^{\alpha+\epsilon} \nu(dy)<\infty$ (for some $\epsilon>0$) then 
\begin{equation}\label{eq:breiman}
\mu\circledast \nu(x,\infty)\sim \int_0^\infty y^\alpha \nu(dy) \mu(x,\infty) \, \, x\to\infty,
\end{equation}
where $\mu\circledast\nu$ denotes the classical multiplicative convolution. A similar result can be obtained when $\nu$ is a regularly varying measure(see~\cite{jessenmikosch}). The result in the case of free multiplicative is still unknown to the best of our knowledge. We provide an example later to show the behaviour is much different from the classical case. In the Boolean convolution, the behaviour turns out to be much similar for multiplicative convolution and in that case again the heavier tail wins. We derive the explicit description in Theorem~\ref{mth1}. The constants appearing though change from the classical case.

 The study of Boolean independence and convolutions although not studied as rigorously as free independence has already established its importance in many areas. The Boolean convolutions of probability measures are also used in studying quantum stochastic calculus, see ~\cite{booleno2}. The Boolean Brownian motion and Poisson processes are investigated using Boolean convolutions to study the Fock space in ~\cite{booleno6}, ~\cite{booleno3}. We can also observe the connection between Appell polynomials and Boolean theory in ~\cite{booleno4}.The Boolean stable laws and their relationship with free and classical stable laws were extensively studied in recent works (see \cite{Booleanariz, Boostableariz2014, Booleanamw, arizmendiTrans}). In a more recent study of classification of easy quantum groups, it was shown the non-commutative analogue of de Finetti's theorem (quantum exchangeability) holds true and the notions of free independence, classical independence and half independence arise in this context (see~\cite{banica:curran:speicher}). The relation between Boolean independence and de Finetti's theorem was recently studied by~\cite{liu:2015}. Recently it has been established that in some random matrices, the asymptotic Boolean independence can arise (\cite{ lenczewski:2011,male2011, popa2017}).

As an application for the above results we determine the behaviour of the Belinschi-Nica map which is a one parameter family of maps $\{\textbf B_t\}_{t\geqslant  0}$ on set of probability measures and was introduced by~\cite{belinicabool} (see precise definition in next section). It is well known that classical infinitely divisible distributions are in bijection with the free infinitely divisible distributions. Here the map $\textbf B_1$ turns out to be a bijection from Boolean to free infinitely divisible distributions. In fact it turns out that $(\textbf B_t(\mu))_{t\geqslant  1}$ is $\boxplus$-infinitely divisible for every probability measure $\mu$. The relationship with free Brownian motion and complex Burgers equation makes it an extremely important object of study. The map was further studied in \cite{Booleanariz}, \cite{Boostableariz2014}. In this article we study the case when $\mu$ is a heavy tail distribution and show that $\mu$ is regularly varying of index $-\alpha$ if and only if $\textbf B_t(\mu)$ is regularly varying $-\alpha$ for $t\geqslant  0$. In particular, it shows that the support of $\textbf B_t(\mu)$ will be unbounded whenever $\mu$ has such regularly varying tails. The Boolean extreme value theory was recently explored in~\cite{vargas2017boolean} in parallel to the study of free extreme value theory (\cite{extremeben}). We show that in the subexponential case, the tail behaviour of Boolean, free and classical extremes are asymptotically equivalent. It is known that the classical subexponential random variables satisfy \textit{the  principle  of  one  large  jump}, that is,  if $\{X_i\}$ are i.i.d. subexponential random variables, then for all $n\ge 1$, 
$$P( \sum_{i=1}^n X_i > x)\sim n P(X_1>x)\sim P(\max_{1\le i\le n} X_i>x) \text{ as } x\to \infty.$$
The free max convolution, denoted by $\boxv$, was introduced in \cite{extremeben} and the analogous result for the free one large jump principle was obtained in \cite{hazra1}. In this article we show that Boolean subexponential distributions follow the principle of one large jump also. 

The main techniques involved in the proof of the above results is to study the transforms and their Taylor series expansion. In particular we show the remainder terms of the respective transforms carries information about the regular variation and also it is preserved under certain operations such as taking a reciprocal. These results can be independent in their own interest and can be used to study various properties of the transforms involved in free and Boolean independence. Such ideas were first explored in the works of \cite{ber:pata:biane} to show the bijection between free infinitely distributions with the classical counter parts. Other works relating the remainder terms of Cauchy and $R$-transforms were studied in \cite{bercovici2000functions, taylorseries, hazra1}.

{\bf Outline of the article:} In the next section we develop the set-up and state our main results precisely. To prove the results we use various transforms and their relationship with the tail of the measures.  In Section \ref{p1section2} we introduce some of the transforms used, an interesting property  of the Belinschi-Nica map $\{\textbf{B}_t\}_{t \geqslant 0}$ in Theorem \ref{theorem:application} and the principle of one large jump in Proposition \ref{prop:oljboolean}. In Section \ref{sectionthree} we state the relationship between the tail of a regularly varying probability measure the remainder of $1/B$-transform followed by defining the remainder terms.  In Section~\ref{sectionfour} we use these relations to provide proof of the results for the additive Boolean convolutions. Section \ref{p1section5} contains the proof of Theorem \ref{mth1} about the multiplicative Boolean convolution. Finally in section \ref{secsce} we prove the technical results which are presented in Section \ref{sectionthree}.


\section{Preliminaries and main results}\label{p1section2}
A real valued measurable function $f$ defined on non-negative real line is called \textit{regularly varying} (at infinity) with index $\alpha$ if for every $t>0$, $\frac{f\left(tx\right)}{f\left(x\right)} \rightarrow t^{\alpha}$ as $x \rightarrow \infty$. If $\alpha = 0$, then $f$ is said to be a slowly varying function (at infinity). Regular variation with index $\alpha$ at zero is defined analogously. In fact, $f$ is regularly varying at zero of index $\alpha$, if the function $x \mapsto f(\frac{1}{x})$ is regularly varying at infinity of index $-\alpha$. {\it Unless otherwise mentioned, the regular variation of a function will be considered at infinity. For regular variation at zero, we shall explicitly mention so}. A distribution function $F$  on $[0,\infty)$ has regularly varying tail of index $-\alpha$ if $\overline{F}(x) = 1 - F(x)$ is regularly varying of index $-\alpha$. Since $\overline{F}(x) \rightarrow 0$ as $x \rightarrow \infty$ , we must necessarily have $\alpha \geqslant 0$. A probability measure on $[0,\infty)$  with regularly varying tail is defined through its distribution function. Equivalently, a measure $\mu$ is said to have a regularly varying tail of index $-\alpha$, if $\mu(x,\infty)$ is regularly varying of index $-\alpha$ as a function of $x$.

The real line and the non-negative part of the real line will be denoted by $\mathbb{R} $ and $ \mathbb{R}_{+}$ respectively. The complex plane will be denoted by $\mathbb C$ and for a complex number $z$, $\Re z$ and $\Im z$ will denote its real and imaginary parts respectively. Given positive numbers $\kappa$ and $\delta$, let us define the following cone:
$$\Delta_{\kappa}=\{z\in \mathbb{C}^-: |\Re z|<-\kappa \Im z\} \text{  and  } \Delta_{\kappa,\delta}=\{z\in\Delta_{\kappa}: |z|<\delta\},$$
where $\mathbb C^+$ and $\mathbb C^-$ are the upper and the lower halves of the complex plane respectively, namely, $\mathbb C^+ = \{z\in\mathbb C: \Im z>0\}$ and $\mathbb C^- = - \mathbb C^+$. Then we shall say that $f\left(z\right) \to l$ as $z$ goes to $0$ n.t. (non-tangentially), if for any $\epsilon>0$ and $\kappa>0$, there exists $\delta\equiv \delta\left(\kappa,\epsilon\right)>0$, such that $|f\left(z\right) - l|< \epsilon$, whenever $z \in \Delta_{\kappa, \delta}$.

We shall write $f\left(z\right)\approx g\left(z\right)$, $f\left(z\right)=\lito\left(g\left(z\right)\right)$ and $f\left(z\right)=\bigo\left(g\left(z\right)\right)$ as $z\to 0$ n.t.\ to mean that $f\left(z\right)/g\left(z\right)$ converges to a non-zero limit, ${f\left(z\right)}/{g\left(z\right)}\rightarrow 0$ and $f\left(z\right)/g\left(z\right)$ stays bounded as $z\to 0$ n.t.\ respectively. If the non-zero limit is $1$ in the first case, we write $f\left(z\right) \sim g\left(z\right)$ as $z\to 0$ n.t. For $f\left(z\right)=\lito\left(g\left(z\right)\right)$ as $z\to 0$ n.t., we shall also use the notations $f\left(z\right)\ll g\left(z\right)$ and $g\left(z\right)\gg f\left(z\right)$ as $z\to 0$ n.t.

$\mathcal{M} $ and $\mathcal{M}_{+} $ are the set of probability measures supported on $\mathbb{R}$ and $\mathbb{R}_{+}$  respectively. By $\mathcal{M}_{p}$ we mean the set of probability measures on $[0, \infty)$ whose $p$-th moment is finite and do not have the $(p+1)$-th moment.

\subsection{Additive Boolean convolution}
For a probability measure $\mu\in \mathcal{M}$, its Cauchy transform is defined as
$$G_{\mu}\left(z\right)=\int_{-\infty}^{\infty}\frac{1}{z-t}d\mu\left(t\right),\ \ \  z\in \mathbb{C}^+.$$
Note that $G_{\mu}$ maps $\mathbb{C}^+$ to $\mathbb{C}^{-}$. The Boolean additive convolution is determined by the transform $K_\mu$ which is defined as
\begin{align} \label{kgrelationhere}
K_{\mu}\left(z\right) &= z - \frac{1}{G_{\mu}\left(z\right)}, \text{  for  } z \in \mathbb{C}^{+}.
\end{align}
For two probability measures $\mu$ and $\nu$, the additive Boolean convolution $\mu \uplus \nu$ is determined by 
\begin{align} \label{Kmu}
K_{\mu \uplus \nu}\left(z\right) = K_{\mu}\left(z\right) + K_{\nu}\left(z\right), \text{  for  } z \in \mathbb{C}^{+}
\end{align}
and $\mu \uplus \nu$ is again a probability measure.

Our first result describes the behaviour of additive Boolean convolution under the regularly varying measures. 
Suppose $\{X_{i}\}_{i \geqslant 1}$ be independent (classically) and identically distributed non-negative regularly varying random variables of index $-\alpha$, $\alpha\geqslant  0$ and denote $S_{n}=X_{1}+X_{2}+ \cdots +X_{n}$. Then it is known that 
\begin{equation}\label{eq:subexp:classical}
P(S_{n} > x) \sim nP(X_{1} > x)\,\, \text{ as } x\to \infty.
\end{equation}
The proof of the above fact can be found in \cite{fellertwobook}. If a sequence of random variables follows~\eqref{eq:subexp:classical} then they are called subexponential. In the case of free additive convolution, the parallel result was shown in \cite{hazra1}, which states:
\begin{align*}
\mu^{\boxplus{n}}\left(y,\infty\right) = \underbrace{\left(\mu \boxplus \cdots \boxplus \mu\right)}_{n\hspace{.15cm}times}\left(y,\infty\right) \sim  n\mu\left(y,\infty\right) \text{  as  } y \rightarrow \infty,
\end{align*}
when $\mu$ has regularly varying tail of index $-\alpha$, $\alpha\geqslant  0$. We show that result can be extended to Boolean additive convolution also. To state the result we first introduce the definition of Boolean subexponentiality:

\begin{definition}
A probability measure $\mu $ on $[0,\infty)$, with $\mu\left(y,\infty\right) > 0 $ for all $y \geqslant 0$, is said to be {\it Boolean-subexponential} if for all $n \in \mathbb{N}$,
\begin{align*}
\mu^{\uplus{n}}\left(y,\infty\right) = \underbrace{\left(\mu \uplus \cdots \uplus \mu\right)}_{n\hspace{.15cm}times}\left(y,\infty\right) \sim  n\mu\left(y,\infty\right) \text{  as  } y \rightarrow \infty.
\end{align*}
\end{definition}
Our first result shows that analogue of the classical and free case is also valid in Boolean set-up. 
\begin{theorem} \label{lastth}
If $\mu $ is regularly varying of index $-\alpha $ with $\alpha\geqslant  0$, then $\mu $ is Boolean-subexponential.
\end{theorem}
The proof uses the relation between $\mu$ and $G_\mu$ developed in~\cite{hazra1} and also extensions to the  transforms $K_\mu$. 
\subsection{Applications of Boolean subexponentiality}
In this subsection we see two important applications of Boolean subexponentiality. As mentioned in the introduction that there are three universal notions of independence and these three notions give rise to corresponding extreme value theory. We first show that subexponentiality in all the three notions are asymptotically equivalent. In the second application we show how the Belinschi-Nica map related to free infinitely divisible indicator behaves for a regularly varying measure. 
\subsubsection{Applications to Boolean extremes} The very immediate upshot of the definition of Boolean subexponentiality is the {\textit principle of one large jump} which gives us the asymptotic relation between the sum and maximum of a finite collection of i.i.d. probability distributions. The extreme value theory in Boolean independence was recently explored in \cite{vargas2017boolean}.  We briefly recall the definition of Boolean max convolution from~\cite{vargas2017boolean}.
\begin{definition}
Let $F_1, F_2$ be two distributions on $[0, \infty)$. Their Boolean max convolution is defined by, $$(F_1 \maxbool F_2)(t) = F_1(t) \minbool F_2(t)$$ where the operation $\minbool$ is defined as $$(x \minbool y)^{-1} -1 = (x^{-1}-1) + (y^{-1}-1) \text{  for all } x,y \in [0,1].$$
\end{definition}   
Let $D_+$ be the set of all probability distributions on $[0, \infty)$. Then $D_+$ forms semigroup with respect to both the classical max convolution \enquote{ $\cdot$ } and the boolean max convolution \enquote{ $\maxbool$ }. Further it is proved there that the map $X:(D_+, \cdot) \to (D_+, \maxbool)$, given by, 
\begin{equation}\label{themapX}
X(F)(t) = \exp\Big(1 - \frac{1}{F(t)}\Big) \text{  for all } t \in [0,\infty), \text{ }F \in D_+
\end{equation}
 is an isomorphism while the inverse map is 
 \begin{equation}\label{themapXinv}
 X^{-1}(F)(t) = \left(1 - \log(F)\right)^{-1}(t) = \frac{1}{1-\log\left(F(t)\right)} \text{  for all } t \in [0,\infty) \text{ }F \in D_+.
 \end{equation}
 The above isomorphism is obtained by observing an interesting isomorphism between the two semigroups $([0,1], \minbool)$ and $([0,1], . )$ where \enquote{ . } is the usual multiplication of real numbers. Here we give an affirmative answer for the one large jump principle in the Boolean case and combining all the results of the classical,  free and Boolean instances we can further say that all the tails of classical, free and Boolean max convolutions are asymptotically equivalent for the class of regularly varying distributions. We shall use the notations $F^{\maxbool n} $ and $F^{\uplus n}$ for the distributions $\underbrace{F \maxbool \cdots \maxbool F}_{n \text{ }times}$ and $\underbrace{F \uplus \cdots \uplus F}_{n \text{ }times}$ respectively.
 \begin{proposition}\label{prop:oljboolean}
 The principle of one large jump holds true for Boolean-subexponential distributions, namely, if $F$ is Boolean-subexponential then for every $n\ge 1$, $$\overline{F^{\uplus n}}(y) \sim \overline{F^{\maxbool n}}(y) \text{  as } y \to \infty.$$ Moreover if $F$ is regularly varying with index $-\alpha$, $\alpha \geqslant 0$, then for all $n\ge 1$,
 \begin{equation}\label{moreoverstat}
 \overline{F^{\maxbool n}}(y) \sim \overline{F^{\boxv n}}(y) \sim \overline{F^{n}}(y) \text{  as } y \to \infty
\end{equation}
 where $F^{n}$ arises out of the classical max convolution of classical independent random variables $Z_1, \cdots , Z_n$ having identical distribution $F$.
 \end{proposition}

\subsubsection{Application to the Belinschi-Nica map}
Before going to the multiplicative Boolean convolution we want to show an application of the above result to the Belinschi-Nica map. Let us recall that $\boxplus$ denotes the free additive convolution of measures.  Let us consider the map $\textbf{B}_t:\mathcal{M} \to \mathcal{M}$ for all $t \geqslant 0$, given by
\begin{align}\label{applidef}
\textbf{B}_t(\mu)=(\mu^{\boxplus(1+t)})^{\uplus\frac{1}{1+t}} \hspace{1cm} \mu \in \mathcal{M}.
\end{align}
This map was introduced in \cite{belinicabool} noting that every probability measure on $\mathbb{R}$ is infinitely divisible with respect to additive Boolean convolution. It was also shown there that if $\mu \in \mathcal{M}_+$ then $\textbf{B}_t(\mu) \in \mathcal{M}_+$. When $t=1$, the map $\textbf{B}_1$ coincides with the Bercovici-Pata bijection between $\mathcal{M}$ and the class of all free infinitely divisible probability measures supported on $\mathbb{R}$. Here for better understanding we can consider the maps $\textbf{B}_n:\mathcal{M}_+ \to \mathcal{M}_+$ for non-negative integers $n$ and from the definition \eqref{applidef}, we have
\begin{align}\label{applithzero}
\textbf{B}_n(\mu)^{\uplus (1+n)}:=\underbrace{\textbf{B}_n(\mu) \uplus \textbf{B}_n(\mu) \uplus \cdots \uplus \textbf{B}_n(\mu)}_{(1+n \hspace{.2cm}times)} = \underbrace{\mu \boxplus \mu \boxplus \cdots \boxplus \mu}_{(1+n \hspace{.2cm}times)}=:\mu^{\boxplus(1+n)}.
\end{align}
\begin{theorem}\label{theorem:application}
The following are equivalent for a probability measure $\mu \in \mathcal{M_+}.$
\begin{enumerate}
\item $\mu$ is regularly varying with tail index $-\alpha$.
\item $\textbf{B}_t(\mu)$ is regularly varying with tail index $-\alpha$, for $t\geqslant  0$. 
\end{enumerate}
Furthermore, if any of the above holds, we also have as $y \to \infty$, $$\mu(y,\infty) \sim \textbf{B}_t(\mu)(y,\infty).$$
\end{theorem}

An interesting connection with complex Burgers equation was established in~\cite{belinicabool} using the following function 
$$h(t,z)= F_{\textbf{B}_t(\mu)}(z)-z, \, \forall t>0, \, \, \forall \, z\in \mathbb C^{+},$$
where $F_\nu$ is the reciprocal of the Cauchy transform. Note that it can also be written as $h(t,z)= -K_{\textbf{B}_t}(z)$. It was shown that $h(t,z)$ satisfies the following complex Burgers equation
$$\frac{\partial h}{\partial t}(t,z)= h(t,z) \frac{\partial h}{\partial z} (t,z).$$
The complex Burgers equation (also known as the free analogue of heat equation) arises naturally due to the connections with free Brownian motion (see~\cite{voiculescu1993analogues}). 
In the following section while proving Theorem~\ref{lastth} we shall study the remainder term in the $K$ transform of a measure $\mu$ and hence from the above result one can easily derive the asymptotic behaviour of the remainder term of $h(t,z)$ (taking the Taylor series expansion in $z$) when $\mu$ has a regularly varying tail. Note that in the power series expansion of $K$, the coefficients, which are also known as Boolean cumulants can be directly computed using the moments recursively. We do not write the details of such applications but it would be clear from the derivations later.

\subsection{Multiplicative Boolean convolution}
Now we define the multiplicative Boolean convolution of two probability measures defined on $\mathbb{R}_+$. For $\mu \in \mathcal{M}_{+}$ the function
\begin{align*}
\Psi_{\mu}\left(z\right) = \int_{\mathbb{R}} \frac{zt}{1-zt} d\mu\left(t\right), \hspace{1cm} z \in \mathbb{C} \setminus \mathbb{R}_+
\end{align*}
is univalent in the left-plane $i\mathbb{C}_+$ and $\Psi_{\mu}\left(i\mathbb{C}_+\right)$ is a region contained in the circle
with diameter $\left(\mu\left({0}\right) - 1, 0\right)$. It is well known that,
\begin{align}
\Psi_{\mu}\left(z^{-1}\right) &= zG_{\mu}\left(z\right) - 1. \label{psiG}
\end{align}
We now recall the following from ~\cite{Booleanamw}. For $\mu \in \mathcal{M}_{+}$, define the $\eta$-transform of $\mu$ as $\eta_{\mu}:\mathbb{C} \setminus \mathbb{R}_{+} \to \mathbb{C} \setminus \mathbb{R}_{+}$
\begin{align}\label{etapsi}
\eta_{\mu}\left(z\right) &= \frac{\Psi_{\mu}\left(z\right)}{1+\Psi_{\mu}\left(z\right)}. 
\end{align}
It is clear that $\mu$ is determined uniquely from the function $\eta_{\mu}$. For $\mu \in \mathcal{M}_{+}  $ it is known that $\eta_{\mu}((-\infty,0)) \subset (-\infty,0)$, $0 = \eta_{\mu}(0^{-}) = \lim_{x \rightarrow 0, x < 0}\eta_{\mu}(x)$, $\eta_{\mu}(\overline{z}) = \overline{\eta_{\mu}(z)}$ for $z \in \mathbb{C} \setminus \mathbb{R}_{+}$. Also $\pi > \mathrm{arg}(\eta_{\mu}(z)) \geqslant \mathrm{arg}(z)$, for $z \in \mathbb{C}^{+}$.

The analytic function 
\begin{align}\label{Bsame}
B_{\mu}\left(z\right) &= \frac{z}{\eta_{\mu}\left(z\right)}
\end{align}
is well defined in the region $z \in \mathbb{C} \setminus \mathbb{R}_{+}$.
Now for $\mu, \nu \in \mathcal{M}_{+} $, their \textit{multiplicative Boolean convolution} $\mu \mboli \nu$ is defined as the unique probability measure in $\mathcal{M}_{+} $ that satisfies
\begin{align} \label{Bmuandnu}
B_{\mu \mboli \nu}(z) = B_{\mu}(z)B_{\nu}(z) \text{ for  } z \in \mathbb{C} \setminus \mathbb{R}_{+}.
\end{align}
Note that for $\mu, \nu \in \mathcal{M_{+}} $ which satisfies 
\begin{center}
\begin{itemize}
\item[(a) ]$\mathrm{arg}(\eta_{\mu}(z)) + \mathrm{arg}(\eta_{\nu}(z))- \mathrm{arg} (z) < \pi$ for $z \in \mathbb{C}^{+} \cup (-\infty, 0)$, and 
\item[(b) ]\textbf{at least one of the first moments of one of the measure $\mu$ or $\nu$ exists finitely}, then $\mu \mboli \nu \in \mathcal{M}_{+} $ is well-defined.
\end{itemize}
\end{center}

In this paper whenever we write the probability measure $\mu \mboli \nu$, it is assumed that the first moment $m(\nu)$ of $\nu$ must exist due to the definition of multiplicative Boolean convolution. The mean exists means it is strictly positive since the measures are supported on the positive half of the real line. When two measures shall have the same regularly varying tail, we will assume that there exists some $c \in (0,\infty)$ such that $\nu(x,\infty) \sim c\mu(x,\infty)$ as $x \to \infty$. The case where this asymptotics of tail sums fails is explained in Remark \ref{remark}. Now here is the main result for multiplicative Boolean convolution: 
\begin{theorem} \label{mth1}
Let $\mu$, $\nu \in \mathcal{M}_+$. If $\mu$ is regularly varying of tail index $-\alpha$ and $\nu$ is regularly varying of tail index $-\beta$ where $\alpha \leqslant \beta$ and $\mu \mboli \nu \in \mathcal{M}_{+} $ then $\mu \mboli \nu$ is also regularly varying with tail index $-\alpha$, furthermore,
\begin{align*}
\mu \mboli \nu\left(y,\infty\right) &\sim  m\left(\nu\right)\mu\left(y,\infty\right) \text{ if } \alpha <\beta,\\
\mu \mboli \nu\left(y,\infty\right) &\sim (1+c)m\left(\nu\right)\mu\left(y,\infty\right), \text{ if } \alpha=\beta,
\end{align*}
where $m(\nu)$ is the mean of $\nu$.
\end{theorem}

Note that the result differs from the classical Breiman's result~\eqref{eq:breiman} in terms of the constants which appear in the tail equivalence relation. In classical case, the $\alpha$-th moment of $\nu$ appears and in multiplicative Boolean the first moment appears only. We end the section with a related open question for free multiplicative convolution. 

\subsection{ Open questions} We list some of the open questions in this subsection before going to the technicalities of the proof. 
\begin{enumerate}
\item Suppose $\mu$ and $\nu$ are in $\mathcal M_{+}$ and have regularly varying tails of index $-\alpha$ and $-\beta$ respectively. Then what is the tail behaviour of $\mu\boxtimes\nu$?  From a result from~\cite[Proposition A4.3]{ber:pata:biane} it follows that if $\mu$ is $\boxplus$ stable of index $1/(1+s)$ and $\nu$ is of index $1/(1+t)$ then $\mu\boxtimes \nu$ is $\boxplus$ stable of index $1/(1+s+t)$. This already shows that the classical Breiman's theorem is not true in free set-up and hence it would be interesting to know what kind of behaviour the $\mu\boxtimes \nu$ distribution inherits. 

\item In a recent work \cite{jacobsen2009} it was shown that if one takes the inverse problem of Breiman's theorem, that is, if one knows that $\mu\circledast \nu$ has a regularly varying tail of index $-\alpha$ then under some necessary and sufficient conditions on $\nu$ can determine that $\mu$ also has regularly varying tail of same index. It would be interesting to explore if such inverse problems can be answered in the free or Boolean set-up.

\item Following \cite{belinicabool} we recall the definition of $\boxplus$-divisibility indicator $\phi(\mu)$ of $\mu$ given by $$\phi(\mu) = \sup\{t \in [0,\infty) : \mu \in \textbf{B}_t(\mathcal{M})\} \in [0,\infty].$$
The Cauchy distribution $\mu_{ca}$ (which has regularly varying tail of index $-1$)  is fixed by the map $B_1$ (which is in fact the Boolean to free Bercovici-Pata bijection). Therefore by definition of $\textbf{B}_1$ we have $\mu_{ca}^{\uplus 2} = \mu_{ca}^{\boxplus 2}$ (moreover $\mu_{ca}^{\uplus t} = \mu_{ca}^{\boxplus t}$ for all $t \geqslant 0$) and this along with the formula $\phi(\textbf{B}_t(\mu)) = \phi(\mu) + t$ allows us to conclude that $\phi(\mu_{ca}) = \infty$ as observed by Belinschi and Nica.  It would be interesting to understand if one can take $\boxplus$-infinitely divisible distributions with regularly varying tails and see if $\phi(\mu)=\infty$ in such cases also. 

\end{enumerate}

The rest of the paper is devoted to proofs of the above results. We first develop the Tauberian type results for different transforms and then apply them to prove the results.

\section{Regular variation of the remainder terms of $\Psi$, $\eta$ and $B$ transform}\label{sectionthree}
In this section we define the remainder term of $B$-transforms and shall see how regular variation is linked to it. These relations will be used to prove the main results. Note that although the $B$ transform is used to define the multiplicative Boolean convolution, it can be used in the analysis of the additive transform by its relation to $K$ transform in the following way. From the relations \eqref{kgrelationhere}, \eqref{psiG}, \eqref{etapsi} and \eqref{Bsame} it follows that,
\begin{align}\label{addiKBmu}
K_{\mu}\left(z^{-1}\right) = \frac{1}{B_{\mu}\left(z\right)}.
\end{align}
For $\mu \in \mathcal{M}_{p}$, following \cite{taylorseries}[Theorem~1.5], we have the following Laurent series like expansion
\begin{align*}
G_{\mu}(1/z) &= \sum_{i=1}^{p+1} m_{i-1}(\mu)z^{i} + o(z^{(p+1)}),  \hspace{.5cm}z \to 0 \text{  n.t.}, \text{  } p \geq 0.
\end{align*}
Therefore using ~\eqref{psiG}, we get
\begin{align*}
\Psi_{\mu}(z) &= \sum_{i=1}^{p} m_{i}(\mu)z^{i} + o(z^{p}) \text{  as } z \to 0 \text{ n.t.},\hspace{1cm} p \geq 1.
\end{align*}
The remainder term $r_{G_{\mu}}(z)$ of the Cauchy transform was defined in ~\cite{hazra1}, given by:
\begin{align}
r_{G_{\mu}}(z) := z^{p+1}(G_{\mu}(z) - \sum_{i=1}^{p+1} m_{i-1}(\mu)z^{-i}). \label{rG}
\end{align}
Using ~\eqref{psiG} and the above expressions we define the remainder term $r_{\Psi_{\mu}}(z)$ of the $\Psi$-transform.
\begin{equation}
  \label{rpsi}
r_{\Psi_{\mu}}(z) := 
\begin{cases}
z^{-p}(\Psi_{\mu}(z) - \sum_{i=1}^{p} m_{i}(\mu)z^{i}), & \text{if } p \geq 1 \\ 
  \Psi_{\mu}(z), & \text{if }  p = 0.
 \end{cases}
\end{equation}
Using ~\eqref{psiG},~\eqref{rG} and \eqref{rpsi}, we get
\begin{equation} \label{GPsi}
r_{G_{\mu}}(z) = r_{\Psi_{\mu}}(z^{-1}).
\end{equation}
Therefore we have from ~\eqref{GPsi},
\begin{align} \label{real}
\Re r_{G_{\mu}}(z) = \Re r_{\Psi_{\mu}}(z^{-1}) \text{  and  }
\Im r_{G_{\mu}}(z) = \Im r_{\Psi_{\mu}}(z^{-1}).  
\end{align}
Thus the remainder terms of $\eta$ and $B$ transforms can be defined analogously. Also from \eqref{Bsame} and the fact that $z$ lies either in the upper half or the lower half of the complex plane, we have, 
\begin{equation*} \label{b1}
\frac{1}{B_{\mu}}\left(z\right) := \frac{1}{B_{\mu}(z)} = \frac{\eta_{\mu}\left(z\right)}{z}.
\end{equation*} \label{b2}
Using $\Psi_{\mu}(z)$ has no constant term in its Taylor series expansion, we have for $p \geqslant 1, \mu \in \mathcal{M}_{p}$,
\begin{equation}
r_{\frac{1}{B_{\mu}}}\left(z\right) = r_{\eta_{\mu}}\left(z\right).
\end{equation}
Equating the real and imaginary part for the above identity we get
\begin{equation}\label{eq:t}
\Re r_{\frac{1}{B_{\mu}}}\left(-iy^{-1}\right) = \Re r_{\eta_{\mu}}\left(-iy^{-1}\right)   \text{   and   }
\Im r_{\frac{1}{B_{\mu}}}\left(-iy^{-1}\right) = \Im r_{\eta_{\mu}}\left(-iy^{-1}\right).   
\end{equation}  
When $p = 0$, that is, $\mu \in \mathcal{M}_{0}$, we write
\begin{align*}
r_{\frac{1}{B_{\mu}}}\left(z\right) = \frac{1}{B_{\mu}}\left(z\right)
\text{  and  }
\frac{1}{B_{\mu}}\left(-iy^{-1}\right) = iy\eta_{\mu}\left(-iy^{-1}\right). 
\end{align*}
We will later see $1/B_{\mu}(z)$ in this case goes to infinity as $z \to 0$ non-tangentially but still we want to say it is a remainder term as it helps in keeping analogy with the other cases notationally.

Let $\mu \in \mathcal{M}_+$ and is regularly varying of tail index $-\alpha$. Then there exists a non-negative integer $p$ such that $\mu \in \mathcal{M}_{p}$. We split this into five cases as follows: (i) p is a positive integer and $\alpha \in (p,p+1)$; (ii) p is a positive integer and $\alpha = p$; (iii) $p = 0$ and $\alpha \in [0,1)$; (iv) $p = 0$ and $\alpha = 1$; (v) $p$ is a natural number and $\alpha = p+1$, giving rise to the following five theorems.

In the following we compare the tails of $\mu$ and the behaviour of the remainder term of the $1/B_\mu$. The proof of these five theorems are deferred to the last section and depends crucially on ideas developed in \cite{hazra1}.

We first consider the case where p is a positive integer and $\alpha \in (p,p+1)$.
\begin{theorem} \label{ccor1.1}
Let $\mu$ be in $\mathcal{M}_{p}, \, p \geqslant 1$ and $p < \alpha < p+1$. The following are equivalent:
\begin{enumerate}
\item $\mu\left(y,\infty\right)$ is regularly varying of index $-\alpha$.
\item $\Im r_{\frac{1}{B_{\mu}}}\left(-iy^{-1}\right)$ is  regularly varying of index $-\left(\alpha-p\right)$ and $$\Re r_{\frac{1}{B_{\mu}}}\left(-iy^{-1}\right) \approx \Im r_{\frac{1}{B_{\mu}}}\left(-iy^{-1}\right) \qquad \text{ as } y\to\infty.$$
\end{enumerate}
If any of the above statements holds, we also have, as $z \rightarrow 0$ n.t., 
\begin{equation} \label{ccor1.1a}
z \ll r_{\frac{1}{B_{\mu}}(z)};
\end{equation}
and as $y \rightarrow \infty$, 
\begin{align} \label{ccor1.1b}
\Im r_{\frac{1}{B_{\mu}}(-iy^{-1})} &\sim -\frac{\pi\left(p+1-\alpha\right)/2}{\cos\left(\pi\left(\alpha-p\right)/2\right)}y^{p}\mu\left(y,\infty\right) \gg y^{-1}
\end{align}
and
\begin{align} \label{ccor1.1c}
\Re r_{\frac{1}{B_{\mu}}(-iy^{-1})} &\sim -\frac{\pi\left(p+2-\alpha\right)/2}{\sin\left(\pi\left(\alpha-p\right)/2\right)}y^{p}\mu\left(y,\infty\right) \gg y^{-1}.
\end{align}
\end{theorem}
Next we consider the case where $p$ is a positive integer and $\alpha = p$. In this case although \eqref{ccor1.1b} holds but the final asymptotic of \eqref{ccor1.1c} need not be true.
\begin{theorem} \label{ccor1.2}
Let $\mu$ be in $\mathcal{M}_{p}, p \geqslant 1$ and $\alpha = p$. The following are equivalent:
\begin{enumerate}
\item $\mu\left(y,\infty\right)$ is regularly varying of index $-p$.
\item $\Im r_{\frac{1}{B_{\mu}}}\left(-iy^{-1}\right)$ is is slowly varying and
\begin{equation} \label{ccor1.2c}
\Re r_{\frac{1}{B_{\mu}}(-iy^{-1})} \gg y^{-1} \qquad \text{ as } y\to\infty.
\end{equation}
\end{enumerate}
If any of the above statements holds, we also have, as $z \rightarrow 0$ n.t., 
\begin{equation*} \label{ccor1.2a}
z \ll r_{\frac{1}{B_{\mu}}(z)};
\end{equation*}
and as $y \rightarrow \infty$ we have,
\begin{equation} \label{ccor1.2b}
\Im r_{\frac{1}{B_{\mu}}(-iy^{-1})} \sim -\frac{\pi}{2}y^{p}\mu\left(y,\infty\right) \gg y^{-1}.
\end{equation}
\end{theorem}
In the third case we consider $\alpha \in [0,1)$.
\begin{theorem} \label{ccor1.3}
Let $\mu$ be in $\mathcal{M}_{0}$ and $0 \leq \alpha < 1$. The following are equivalent:
\begin{enumerate}
\item $\mu\left(y,\infty\right)$ is regularly varying of index $-\alpha$.
\item $\Im \frac{1}{B_{\mu}}\left(-iy^{-1}\right)$ is regularly varying of index $-(\alpha-1)$ and $$\Re \frac{1}{B_{\mu}}\left(-iy^{-1}\right) \approx \Im \frac{1}{B_{\mu}}\left(-iy^{-1}\right) \qquad \text{ as } y\to\infty$$
\end{enumerate}
If any of the above statements holds, we also have, as $z \rightarrow 0$ n.t., 
\begin{equation*} \label{ccor1.3a}
z \ll z{\frac{1}{B_{\mu}}(z)};
\end{equation*}

and as $y \rightarrow \infty$ we have,
\begin{align} \label{ccor1.3b}
-y^{-1}\Re \frac{1}{B_{\mu}}(-iy^{-1}) &\sim -\frac{\pi\left(1-\alpha\right)/2}{\cos\left(\pi\alpha/2\right)}\mu\left(y,\infty\right) \gg y^{-1}
\end{align}
and
\begin{equation} \label{ccor1.3c}
y^{-1}\Im \frac{1}{B_{\mu}}(-iy^{-1}) \sim -d_{\alpha}\mu\left(y,\infty\right) \gg y^{-1}
\end{equation}
where
\begin{equation*} \label{dalpha}
d_{\alpha}=
\begin{cases}
\frac{\pi\left(2-\alpha\right)/2}{\sin\left(\pi\alpha/2\right)}, & \text{when }\alpha>0,\\
1,                                         & \text{when }\alpha=0.
\end{cases}
\end{equation*}
\end{theorem}
In the fourth case we consider $\alpha = 1$ and $p= 0$.
\begin{theorem} \label{ccor1.4}
Let $\mu$ be in $\mathcal{M}_{0}$ and $\alpha = 1, r \in \left(0,1/2\right)$. The following are equivalent:
\begin{enumerate}
\item $\mu\left(y,\infty\right)$ is regularly varying of index $-1$.
\item $\Im \frac{1}{B_{\mu}}\left(-iy^{-1}\right)$ is slowly regularly varying and
\begin{equation} \label{ccor1.4c}
y^{-1} \ll -y^{-1}\Re \frac{1}{B_{\mu}}\left(-iy^{-1}\right) \ll y^{-\left(1-r/2\right)}.
\end{equation}
\end{enumerate}
If any of the above statements holds, we also have, as $z \rightarrow 0$ n.t., 
\begin{equation*} \label{ccor1.4a}
z \ll  z\frac{1}{B_{\mu}}\left(z\right);
\end{equation*}
and as $y \rightarrow \infty$ we have,
\begin{align} \label{ccor1.4b}
y^{-(1+r/2)} \ll  y^{-1}\Im \frac{1}{B_{\mu}}\left(-iy^{-1}\right) &\sim -\frac{\pi}{2}\mu\left(y,\infty\right) \ll y^{-1+r/2}
\end{align}
\end{theorem}
Finally, we consider the case where $p \geqslant 1$ and $\alpha = p+1$.
\begin{theorem} \label{ccor1.5}
Let $\mu$ be in $\mathcal{M}_{p}$, $p \geqslant 1$ and $ \alpha = p+1, r \in \left(0,1/2\right)$. The following are equivalent:
\begin{enumerate}
\item $\mu\left(y,\infty\right)$ is regularly varying of index $-\left(p+1\right)$.
\item $\Re r_{\frac{1}{B_{\mu}}}\left(-iy^{-1}\right)$ is regularly varying of index $-1$ and
\begin{equation} \label{ccor1.5c}
y^{-1} \ll \Im r_{\frac{1}{B_{\mu}}}\left(-iy^{-1}\right) \ll y^{-\left(1-r/2\right)}.
\end{equation}
\end{enumerate}
If any of the above statements holds, we also have, as $z \rightarrow 0$ n.t., 
\begin{equation} \label{ccor1.5a}
z \ll r_{\frac{1}{B_{\mu}}}\left(z\right);
\end{equation}
As $y \rightarrow \infty$ we have,
\begin{align} \label{ccor1.5b}
y^{-(1+r/2)} \ll \Re r_{\frac{1}{B_{\mu}}}\left(-iy^{-1}\right) &\sim -\frac{\pi}{2}y^{p}\mu\left(y,\infty\right) \ll y^{-\left(1-r/2\right)}.
\end{align}
\end{theorem}
The proof of these theorems are deferred to Section~\ref{secsce}.
\section{Additive Boolean subexponentiality} \label{sectionfour}
To prove the Theorem~\ref{lastth} we need the following lemma.
\begin{lemma} \label{llema}
Suppose $\mu $ and $\nu $ are two probability measures in $[0,\infty)$
 with regularly varying tails of index $-\alpha$ and suppose $\nu(y,\infty) \sim c \mu(y,\infty)$ for some $c > 0$. Then 
\begin{align*}
\mu \uplus \nu\left(y,\infty\right) \sim \left(1+c\right)\mu\left(y,\infty\right) \text{    as   } y \rightarrow \infty.
\end{align*} 
\end{lemma}
\begin{proof}
Depending on where the index $\alpha \geqslant 0$ lies, the proof can be split into five cases as described in Section~\ref{sectionthree}. We shall present the proof for the case when $p \geq 1$ with $\mu \in \mathcal{M}_p$ and $\alpha \in (p,p+1)$. We shall use Theorem~\ref{ccor1.1} to derive this case. The other cases can be dealt in exactly similar fashion using the four other results stated Section~\ref{sectionthree}. Using the relation \eqref{addiKBmu} we define the remainder term of the $K$-transform in the following obvious way:
\begin{equation}\label{KB}
r_{K_{\mu}} \left(\frac{1}{z}\right) = r_{\frac{1}{B_{\mu}}} (z).
\end{equation}

For  $\alpha \in \left(p,p+1\right)$ using Theorem~\eqref{ccor1.1} and taking imaginary and real parts of~\eqref{KB},  we have
\begin{equation}\label{eq:kmu_remainder}
\Im r_{K_{\mu}}(iy) \sim -\frac{\pi(p+1-\alpha)/2}{\cos(\pi(\alpha-p)/2)}y^{p}\mu(y,\infty) \text{ and } \Re r_{K_{\mu}}(iy) \sim -\frac{\pi(p+2-\alpha)/2}{\sin(\pi(\alpha-p)/2)}y^{p}\mu(y,\infty)
\end{equation}
respectively.

An analogous equation for measure $\nu$ can be derived with $\mu$ being replaced by $\nu$ in~\eqref{eq:kmu_remainder}.  Now the equation ~\eqref{Kmu} and the definition of the remainder term gives, $r_{K_{\mu \uplus \nu}}\left(z\right) = r_{K_{\mu}}\left(z\right) + r_{K_{\nu}}\left(z\right)$. Therefore 
\begin{align} \label{442}
\Im r_{K_{\mu \uplus \nu}}\left(iy\right) &\sim -\frac{\pi\left(p+1-\alpha\right)/2}{\cos\left(\pi\left(\alpha-p\right)/2\right)}\left(1 + c\right)y^{p}\mu\left(y,\infty\right)  \text{  as  } y \rightarrow \infty,\\
\Re r_{K_{\mu \uplus \nu}}\left(iy\right) &\sim -\frac{\pi\left(p+2-\alpha\right)/2}{\sin\left(\pi\left(\alpha-p\right)/2\right)}\left(1 + c\right)y^{p}\mu\left(y,\infty\right)  \text{  as  } y \rightarrow \infty, \nonumber
\end{align}
which are regularly varying of index $-\left(\alpha - p\right)$ with $\Im r_{K_{\mu \uplus \nu}}\left(iy\right) \approx \Re r_{K_{\mu \uplus \nu}}\left(iy\right)$ and we conclude $\mu \uplus \nu \in \mathcal{M}_{p}$ by looking at the remainder term of $K_{\mu \uplus \nu}$. Now again using Theorem ~\ref{ccor1.1} for the measure $\mu \uplus \nu$, we have
\begin{align} \label{443}
\Im r_{K_{\mu \uplus \nu}}\left(iy\right) &\sim -\frac{\pi\left(p+1-\alpha\right)/2}{\cos\left(\pi\left(\alpha-p\right)/2\right)}y^{p}\mu \uplus \nu\left(y,\infty\right)  \text{  as  } y \rightarrow \infty.
\end{align} 
Combining ~\eqref{442} and ~\eqref{443} the result follows.
\end{proof}

The proof of Theorem~\ref{lastth} is immediate using induction which we briefly indicate below. 

\textsc{proof of theorem}~\ref{lastth} : Let $\mu $ be regularly varying of tail index $-\alpha$ and supported on $[0,\infty)$. We prove 
\begin{align} \label{Booleanmu}
\mu^{\uplus{n}}\left(y,\infty\right) \sim  n\mu\left(y,\infty\right) \text{  as  } y \rightarrow \infty.
\end{align}
by induction on n. For $n=2$  ~\eqref{Booleanmu} follows from the Lemma~\ref{llema} with both the measures as $\mu$ and $c=1$. To prove ~\eqref{Booleanmu} for $n=m+1$ assuming $n=m$ we take $c=m$ and $\nu = \mu^{\uplus {n}}$ in Lemma~\ref{llema}.\qed

\subsection{Proof of Proposition~\ref{prop:oljboolean} and Theorem~\ref{theorem:application}}
 \begin{proof} [ Proof of Proposition~\ref{prop:oljboolean}]
 Using \eqref{themapX} and \eqref{themapXinv} we have for any $y \in [0, \infty)$,
 \begin{align*}
 F^{\maxbool n}(y) &= X^{-1}\big((X(F(y))^n\big) \\
 &= X^{-1}\Bigg(\exp\Big(n - \frac{n}{F(y)}\Big)\Bigg) \\
 &= \frac{1}{1 - \log\Bigg(\exp\Big(n - \frac{n}{F(y)}\Big)\Bigg)} \\
 &= \frac{F(y)}{n - (n-1)F(y)}.
\end{align*}  
Thus for any $y \geqslant 0$, $$\overline{F^{\maxbool n}}(y) = 1- F^{\maxbool n}(y) = 1- \frac{F(y)}{n - (n-1)F(y)} = \frac{n \overline{F}(y)}{1 + (n-1)\overline{F}(y)}.$$
Now noting the fact that $\overline{F}(y) \to 0$ as $y \to \infty$, we have as $y \to \infty$
\begin{equation}\label{needforpf}
\overline{F^{\maxbool n}}(y) \sim n \overline{F}(y) \sim \overline{F^{\uplus n}}(y).
\end{equation}
The last asymptotic follows by the definition of Boolean-subexponentiality.

The asymptic \eqref{moreoverstat} follows by combining \eqref{needforpf}, Proposition $1.1$ of \cite{hazra1}( in particular for any $n \in \mathbb{N}$, $\overline{F^{\boxv n}}(y) \sim n \overline{F}(y)$ as $y \to \infty$), Lemma $3.8$ of \cite{jessenmikosch} and the fact that regularly varying distributions are classical, free and Boolean subexponential.
\end{proof}
\begin{proof}[ Proof of Theorem~\ref{theorem:application}]
The proof is obvious for $t=0$. Initially we shall prove the result for integer points $t=n$, $n \in \mathbb{N}$. After that we shall extend the proof for any non-negative real number $t$. We start with letting $\mu$ to be regularly varying of tail index $-\alpha$. Again to keep the exposition simple we derive the result when $\alpha \in (p,p+1)$ for some $p\in \mathbb{N}$. In the other cases the result follows by similar argument using corresponding asymptotic relations.

We have from \eqref{applithzero},
$$K_{\textbf{B}_n(\mu)^{\uplus (1+n)}}(iy) = K_{\mu^{\boxplus(1+n)}}(iy).$$
Therefore using \eqref{Kmu} we can write the above expression as
\begin{align} \label{applithone}
(1+n)K_{\textbf{B}_n(\mu)}(iy) = K_{\mu^{\boxplus(1+n)}}(iy). 
\end{align}
Since $\mu$ has regularly varying tail of index $-\alpha$, from Theorem 1.1 of \cite{hazra1} we can conclude that $\mu$ is free subexponential, i.e. 
\begin{align} \label{appliththree}
\mu^{\boxplus(1+n)}(y,\infty) \sim (1+n)\mu(y,\infty) \text{   as   } y \to \infty,
\end{align}
which shows that the probability measure $\mu^{\boxplus(1+n)} \in \mathcal{M}_+$ is also regularly varying of index $-\alpha$ at infinity.
Therefore, using the relation between the remainder term of $K$ and $B$ transforms (see \eqref{KB}) and Theorem \ref{ccor1.1}  we get as $y \to \infty$,
\begin{align} \label{applithtwo}
\Im r_{K_{\mu^{\boxplus(1+n)}}}(iy) &\sim -\frac{\pi(p+1-\alpha)/2}{\cos(\pi(\alpha-p)/2)}y^{p}\mu^{\boxplus(1+n)}(y,\infty),
\end{align}
 which is regularly varying of index $-(\alpha-p)$. Now taking imaginary parts of the remainder terms on both sides of \eqref{applithone} and using \eqref{applithtwo}, we get
\begin{align*}
(1+n)\Im r_{K_{\textbf{B}_n(\mu)}}(iy) &= \Im r_{K_{\mu^{\boxplus(1+n)}}}(iy)\\
&\sim -\frac{\pi(p+1-\alpha)/2}{\cos(\pi(\alpha-p)/2)}y^{p}\mu^{\boxplus(1+n)}(y,\infty)\\
&\overset{\eqref{appliththree}}{\sim} -\frac{\pi(p+1-\alpha)/2}{\cos(\pi(\alpha-p)/2)}(1+n)y^{p}\mu(y,\infty),
\end{align*} 
Therefore 
\begin{align} \label{applithfour}
\Im r_{K_{\textbf{B}_n(\mu)}}(iy) \sim -\frac{\pi(p+1-\alpha)/2}{\cos(\pi(\alpha-p)/2)}y^{p}\mu(y,\infty),
\end{align}
which is again regularly varying of index $-(\alpha - p)$. Similar calculations by taking the real part in place of imaginary part gives $\Re r_{K_{\textbf{B}_n(\mu)}}(iy)$ is regularly varying of index $-(\alpha - p)$, in particular we shall then have $\Im r_{K_{\textbf{B}_n(\mu)}}(iy) \approx \Re r_{K_{\textbf{B}_n(\mu)}}(iy)$. Now the definition of both additive Boolean convolution and the map $\textbf{B}$ allows us to deduce that $\textbf{B}_n(\mu) \in \mathcal{M}_p$ if and only if $\mu \in \mathcal{M}_p$. Thus applying \eqref{ccor1.1b} we get $\textbf{B}_n(\mu)$ is regularly varying of index $-\alpha$ and 
\begin{align} \label{applithfive}
\Im r_{K_{\textbf{B}_n(\mu)}}(iy) \sim -\frac{\pi(p+1-\alpha)/2}{\cos(\pi(\alpha-p)/2)}y^{p}\textbf{B}_n(\mu)(y,\infty).
\end{align}
Hence from \eqref{applithfour} and \eqref{applithfive} we have as $y \to \infty$,
$$\mu(y,\infty) \sim \textbf{B}_n(\mu)(y,\infty).$$
Conversely, suppose that $\textbf{B}_n(\mu)$ is regularly varying of index $-\alpha$. Here also we further suppose that $\alpha \in (p,p+1)$ with $p \geqslant 0$. When $\alpha = p$ or $\alpha = p+1$ the conclusion can be made by using similar arguments and corresponding asymptotic relationships from \cite{hazra1}. Now using Theorem~\ref{lastth} we have, $\textbf{B}_n(\mu)$ is Boolean-subexponential, i.e.;
\begin{align} \label{applithsix}
\textbf{B}_n(\mu)^{\uplus(1+n)}(y,\infty) \sim (1+n)\textbf{B}_n(\mu)(y,\infty),
\end{align}
which also shows that $\textbf{B}_n(\mu)^{\uplus(1+n)}$ is regularly varying of tail index $-\alpha$.  Let us recall the Voiculescu transform of a measure $\mu$. It is known from \cite{bercovici1993free} that $F_{\mu} = 1/G_{\mu}$ has a left inverse $F_{\mu}^{-1}$ (defined on a suitable domain) and $\phi_{\mu}(z) = F_{\mu}^{-1}(z) - z$. For probability measures $\mu$ and $\nu$ one has $\phi_{\mu \boxplus \nu}(z) = \phi_{\mu}(z) + \phi_{\nu}(z)$ on an appropriate domain. The asymptotics of remainder of $\phi_{\mu}$ were derived in \cite{hazra1}. We can write from \eqref{applithzero},
\begin{align}\label{applithseven}
\Im r_{\phi_{\mu^{\boxplus(1+n)}}}(iy) = \Im r_{\phi_{\textbf{B}_n(\mu)^{\uplus(1+n)}}}(iy),
\end{align}
where $r_{\phi_{\mu}}(z)$ is the remainder term of the Voiculescu transform of $\mu$. Now using the fact that $\textbf{B}_n(\mu)^{\uplus(1+n)}$ is regularly varying of tail index $-\alpha$, we get by applying Theorem 2.1 of \cite{hazra1} $\Im r_{\phi_{\textbf{B}_n(\mu)^{\uplus(1+n)}}}(iy)$ is regularly varying of index $-(\alpha -p)$ and 
\begin{align} \label{applitheight}
\Im r_{\phi_{\textbf{B}_n(\mu)^{\uplus(1+n)}}}(iy) &\sim -\frac{\pi(p+1-\alpha)/2}{\cos(\pi(\alpha-p)/2)}y^{p}\textbf{B}_n(\mu)^{\uplus(1+n)}(y,\infty),\nonumber \\
&\overset{\eqref{applithsix}}\sim -\frac{\pi(p+1-\alpha)/2}{\cos(\pi(\alpha-p)/2)}y^{p}(1+n)\textbf{B}_n(\mu)(y,\infty).
\end{align}
Now combining \eqref{applithseven}, \eqref{applitheight} and using the fact that $r_{\phi_{\mu\boxplus\nu}}(z) = r_{\phi_{\mu}}(z) + r_{\phi_{\nu}}(z)$, we have
\begin{align}\label{applithnine}
\Im r_{\phi_{\mu}}(iy) \sim -\frac{\pi(p+1-\alpha)/2}{\cos(\pi(\alpha-p)/2)}y^{p}\textbf{B}_n(\mu)(y,\infty).
\end{align}
This shows that $\Im r_{\phi_{\mu}}(iy)$ is regularly varying of index $-(\alpha-p)$ and again applying Theorem 2.1 of \cite{hazra1}, we get
\begin{align}\label{applithten}
\Im r_{\phi_{\mu}}(iy) \sim -\frac{\pi(p+1-\alpha)/2}{\cos(\pi(\alpha-p)/2)}y^{p}\mu(y,\infty).
\end{align}
Combining \eqref{applithnine} and \eqref{applithten} it follows $\mu$ is regularly varying of tail index $-\alpha$ and as $y \to \infty$,
$$\mu(y,\infty) \sim \textbf{B}_n(\mu)(y,\infty).$$
Therefore we are done for the integer case. Further we recall the definitions of $\mu^{\boxplus t}$ and $\mu^{\uplus t}$ form \cite{belinicabool} (See also \cite{rezspeibooli} and \cite{atombelinschi} for more details):

For any $t \geqslant 1$ and $\mu \in \mathcal{M}_+$ there exists $\mu^{\boxplus t} \in \mathcal{M}_+$ satisfying $\phi_{\mu^{\boxplus t} }(z) = t \phi_{\mu}(z)$ and thus 
\begin{equation}\label{p1somikoron1}
r_{\phi_{\mu^{\boxplus t} }(z)} = t r_{\phi_{\mu}}(z)
\end{equation}
on a truncated angular domain (e.g. $z$ with $1/z \in \Delta_{\kappa,\delta}$ for some positive $\kappa$, $\delta$) where they are well defined. Also for any $t \geqslant 0$ and $\mu \in \mathcal{M}_+$ there exists $\mu^{\uplus t} \in \mathcal{M}_+$ such that $K_{\mu^{\uplus t} }(z) = t K_{\mu}(z)$ on the upper half plane. So 
\begin{equation}\label{p1somikoron2}
r_{K_{\mu^{\uplus t} }(z)} = t r_{K_{\mu}}(z).
\end{equation}

Now successive use of \eqref{p1somikoron1}, \cite{hazra1}[Theorem~2.1-2.4] according to suitable cases, the fact  $\mu \in \mathcal{M}_p$ implies $\mu^{\boxplus t} \in \mathcal{M}_p$ which follows from \cite{taylorseries}[Theorem~1.5] and the definition of $K$ transform, we can say that any probability measure $\mu$ with regularly varying tail of index $-\alpha$, $\alpha \geqslant 0$ is more than free subexponential, i.e.,

$1$) If $\mu$ is regularly varying of tail index $-\alpha$ then $\mu^{\boxplus t}$ is also so for all $t \geqslant 1$. In particular, as $y \to \infty$, we have $\mu^{\boxplus t}(y, \infty) \sim t \mu(y,\infty).$

Also using \eqref{p1somikoron2}, Theorem~\ref{ccor1.1}-\ref{ccor1.5} for respective cases and the fact that $\mu^{\uplus t} \in \mathcal{M}_p$ whenever $\mu \in \mathcal{M}_p$ we shall be able to conclude that

$2$) If $\mu$ is regularly varying of tail index $-\alpha$ then $\mu^{\uplus t}$ is also so for all $t \geqslant 0$. In particular, as $y \to \infty$, we have $\mu^{\uplus t}(y, \infty) \sim t \mu(y,\infty).$

Now the above facts and similar calculations done in the above proof for the integer case gives us the result for any non-negative real number $t$. 
\end{proof}
\section{Result on multiplicative Boolean convolution} \label{p1section5}
In this short section we will prove Theorem \ref{mth1} using the relation of $B$-transform with the multiplicative Boolean convolution. We begin by observing that when $\mu$ has $p$  moments and $\nu$ has $q$ moments, then the Boolean multiplicative convolution of $\mu$ and $\nu$ has exactly $p$ moments if $p \leqslant q$. 
\begin{lemma} \label{mnbowtie}
Suppose $p \leqslant q$, $\mu \in \mathcal{M}_{p}$ and $\nu \in \mathcal{M}_{q}$, then $\mu \mboli \nu \in \mathcal{M}_{p}.$
\end{lemma}
\begin{proof}
Note that when $\mu$ has infinite mean the result is obvious. Now suppose $\mu$ and $\nu$ both have finite mean then from the definition of the remainder terms we can write the $B$-transforms of $\mu$ and $\nu$ in the following way
\begin{align*}
f_{1}\left(z\right) := \frac{1}{m(\mu)B_{\mu}}\left(z\right) &= 1 + c_{1}z + \cdots + c_{p-1}z^{p-1} + z^{p-1}r_{f_{1}}\left(z\right) \\
\end{align*} where $c_{i}, i = 1,2, \cdots ,p-1 $ are real constants and
\begin{align*}
f_{2}\left(z\right) := \frac{1}{m(\nu)B_{\nu}}\left(z\right) = 1 + d_{1}z + \cdots +d_{p-1}z^{p-1}
+ \cdots + d_{q-1}z^{q-1} + z^{q-1}r_{f_{2}}\left(z\right)
\end{align*}
where $ d_{j}, j = 1,2, \cdots q-1$ are also real constants. Taking the product of $f_{1}\left(z\right)$ and   $f_{2}\left(z\right)$ we see that $\frac{1}{B}$-transform of $\mu \mboli \nu$ has a Taylor series expansion of order $p-1$ by \eqref{Bmuandnu}.  Now it is easy to see that $\mu \in \mathcal{M}_{p}$ is equivalent to $\frac{1}{B_{\mu}}(z)$ having a Taylor series expansion of order $p-1$ (see Theorem 1.5 of \cite{taylorseries}). Therefore we see that $\mu \mboli \nu \in \mathcal{M}_{p}$.
\end{proof}
Following equation \eqref{Bmuandnu} and using the results in Section \ref{sectionthree} we shall derive the relation between the real or imaginary parts of the remainder terms of the product of two $B$-transforms. The Theorem is split into several cases depending on the existence of integer moments of the two measures involved.  
\begin{theorem} \label{Bmuz} 
Suppose $\alpha \leqslant \beta $ and let $\mu$ and $\nu$ be regularly varying with indices $-\alpha$ and $-\beta$ respectively. So there exists a non-negative integer $p$ such that $\alpha \in [p,p+1]$ and $\mu \in \mathcal{M}_{p}$. We also suppose that $\nu$ has finite first moment.\footnote{this is needed to define multiplicative Boolean convolution} Then we have the following:
\begin{enumerate}
\item Suppose $\alpha < \beta$. 
\begin{enumerate}
\item \label{theoremoa}If $1 \leqslant p < \alpha < p+1$ or $\alpha \in (0,1)$, then as $y \to \infty$ 
\begin{align*}
\Im r_{\frac{1}{B_{\mu}B_{\nu}}}\left(-iy^{-1}\right) &\sim m\left(\nu\right)\Im r_{\frac{1}{B_{\mu}}}\left(-iy^{-1}\right) \text {   and   }\\
\Re r_{\frac{1}{B_{\mu}B_{\nu}}}\left(-iy^{-1}\right) &\sim m\left(\nu\right)\Re r_{\frac{1}{B_{\mu}}}\left(-iy^{-1}\right).
\end{align*}
\item \label{theoremob} If $p \geqslant 1$, $\alpha = p$, then as $y \to \infty$ 
\begin{align*}
\Im r_{\frac{1}{B_{\mu}B_{\nu}}}\left(-iy^{-1}\right) &\sim m\left(\nu\right)\Im r_{\frac{1}{B_{\mu}}}\left(-iy^{-1}\right) \text {   and   }\\
\Re r_{\frac{1}{B_{\mu}B_{\nu}}}\left(-iy^{-1}\right) &\gg y^{-1}.
\end{align*}
\item \label{theoremoc}If $p=0$, $\alpha = 1$, $r \in (0,1/2)$, then as $y \to \infty$ 
\begin{align*}
\Im r_{\frac{1}{B_{\mu}B_{\nu}}}\left(-iy^{-1}\right) \sim m\left(\nu\right)\Im r_{\frac{1}{B_{\mu}}}\left(-iy^{-1}\right) \text {   and   }\\
y^{-1} \ll -y^{-1}\Re r_{\frac{1}{B_{\mu}B_{\nu}}}\left(-iy^{-1}\right) \ll y^{-(1-r/2)}.
\end{align*} 
\item \label{theoremod}If $p \geqslant 1$, $\alpha = p+1$, $r \in (0,1/2)$, then as $y \to \infty$ 
\begin{align*}
\Re r_{\frac{1}{B_{\mu}B_{\nu}}}\left(-iy^{-1}\right) \sim m\left(\nu\right)\Re r_{\frac{1}{B_{\mu}}}\left(-iy^{-1}\right) \text {   and   }\\
y^{-1} \ll \Im r_{\frac{1}{B_{\mu}B_{\nu}}}\left(-iy^{-1}\right) \ll y^{-(1-r/2)}.
\end{align*}
\end{enumerate}
\item \label{theoremoe}Suppose $\alpha = \beta$ and there exists some $c \in (0,\infty)$ such that $\nu(x,\infty) \sim c\mu(x,\infty)$. Then \eqref{theoremoa}(with only $p \geqslant1$), \eqref{theoremob} and \eqref{theoremod} holds with $m(\nu)$ is replaced by $(1+c)m(\nu)$ at each places.
\end{enumerate}
\end{theorem}
We shall provide a detailed proof of this result in Section~\ref{secsce}.
We prove the main Theorem \ref{mth1} for multiplicative Boolean convolution using the above result.

\textbf{Proof of Theorem ~\ref{mth1}:} Suppose $\mu \in \mathcal{M}_{p}$, $1 \leqslant p < \alpha < p+1$ and $\nu \in \mathcal{M}_{q}$, $q \geqslant p$ with $\alpha < \beta$ then using \eqref{Bmuandnu}
and \eqref{theoremoa} of Theorem~\ref{Bmuz} we have as $y \to \infty$,
\begin{align*}
\Im r_{\frac{1}{B_{\mu\mboli\nu}}}\left(-iy^{-1}\right)  \sim m\left(\nu\right)\Im r_{\frac{1}{B_{\mu}}}\left(-iy^{-1}\right) \text{    and   }
\Re r_{\frac{1}{B_{\mu\mboli\nu}}}\left(-iy^{-1}\right)  \sim m\left(\nu\right)\Re r_{\frac{1}{B_{\mu}}}\left(-iy^{-1}\right).
\end{align*}
Therefore using \eqref{ccor1.1b}, \eqref{ccor1.1c} and above asymptotics, we get 
\begin{align} \label{435}
\Im r_{\frac{1}{B_{\mu\mboli\nu}}}\left(-iy^{-1}\right) &\sim -\frac{\pi\left(p+1-\alpha\right)/2}{\cos\left(\pi\left(\alpha-p\right)/2\right)}y^{p}m\left(\nu\right)\mu\left(y,\infty\right) \text{    and}\\
\Re r_{\frac{1}{B_{\mu\mboli\nu}}}\left(-iy^{-1}\right) &\sim -\frac{\pi\left(p+2-\alpha\right)/2}{\sin\left(\pi\left(\alpha-p\right)/2\right)}y^{p}m\left(\nu\right)\mu\left(y,\infty\right) \text{    as } y \to \infty \nonumber. 
\end{align}
So $\Im r_{\frac{1}{B_{\mu\mboli\nu}}}\left(-iy^{-1}\right) \approx \Re r_{\frac{1}{B_{\mu\mboli\nu}}}\left(-iy^{-1}\right)$ are both regularly varying $-\left(\alpha - p\right)$. By Lemma~\ref{mnbowtie}, we have $\mu \mboli \nu \in \mathcal{M}_{p}$ and therefore by applying the reverse implication of Theorem~\ref{ccor1.1} we get $\mu \mboli \nu$ is regularly varying of index $-\alpha$ and again using the asymptotic equivalence ~\eqref{ccor1.1b} of Theorem~\ref{ccor1.1} for the measure $\mu\mboli\nu$ we have, 
\begin{align} \label{436}
\Im r_{\frac{1}{B_{\mu\mboli\nu}}}\left(-iy^{-1}\right) \sim -\frac{\pi\left(p+1-\alpha\right)/2}{\cos\left(\pi\left(\alpha-p\right)/2\right)}y^{p}\mu \mboli \nu\left(y,\infty\right).
\end{align}
Hence from ~\eqref{435} and ~\eqref{436} we get $\mu \mboli \nu\left(y,\infty\right) \sim m\left(\nu\right)\mu\left(y,\infty\right)$. The other cases can be similarly dealt with using Theorem \ref{Bmuz} and the remaining four theorems in Section \ref{sectionthree}. We skip the details.
\begin{remark} \label{remark}
As we have mentioned in the beginning of Theorem \ref{mth1}, we have not dealt with the case when $\mu$ and $\nu$ both have the same regularly varying tail index but they are not tail balanced which can happen only in the case when $p \geqslant 1$, $\mu$ is in $\mathcal{M}_{p}$ and $\nu$ is in $\mathcal{M}_{p+1}$ but they are both regularly varying of tail index $-(p+1)$. In this case similar calculations will show that $\mu \mboli \nu\left(y,\infty\right) \sim m\left(\nu\right)\mu\left(y,\infty\right)$, i.e., the constant in the left of the asymptotic is $1$ instead of the form $1+c$.   
\end{remark}
\section{Proofs of Theorem \ref{ccor1.1} to \ref{ccor1.5} and Theorem \ref{Bmuz}}\label{secsce}
To keep the paper self contained we recall the following results from \cite{hazra1}[Theorem 2.1 - 2.4]. The results there gave the relation between $\mu$ and the remainder of Cauchy transform. We use the equation~\eqref{real} to rewrite them in terms of remainder of $\Psi_{\mu}$:

When p is a positive integer and $\alpha \in (p,p+1)$.
\begin{theorem} \label{cor1.1}
Let $\mu$ be in $\mathcal{M}_{p}, p \geqslant 1$ and $p < \alpha < p+1$. The following are equivalent:
\begin{enumerate}
\item $\mu\left(y,\infty\right)$ is regularly varying of index $-\alpha$.
\item $\Im r_{\Psi_{\mu}}\left(-iy^{-1}\right)$ is regularly varying of index $-\left(\alpha-p\right)$ and $$\Re r_{\Psi_{\mu}}\left(-iy^{-1}\right) \approx \Im r_{\Psi_{\mu}}\left(-iy^{-1}\right).$$
\end{enumerate}
If any of the above statements holds, we also have, as $z \rightarrow 0$ n.t., 
\begin{equation} \label{cor1.1a}
z \ll r_{\Psi_{\mu}}(z);
\end{equation}
as $y \rightarrow \infty$
\begin{align} \label{cor1.1b}
\Im r_{\Psi_{\mu}}\left(-iy^{-1}\right) &\sim -\frac{\pi\left(p+1-\alpha\right)/2}{\cos\left(\pi\left(\alpha-p\right)/2\right)}y^{p}\mu\left(y,\infty\right) \gg y^{-1}
\end{align}
and
\begin{align} \label{cor1.1c}
\Re r_{\Psi_{\mu}}\left(-iy^{-1}\right) &\sim -\frac{\pi\left(p+2-\alpha\right)/2}{\sin\left(\pi\left(\alpha-p\right)/2\right)}y^{p}\mu\left(y,\infty\right) \gg y^{-1}.
\end{align}
\end{theorem}
When $p$ is a positive integer and $\alpha = p$.
\begin{theorem} \label{cor1.2}
Let $\mu$ be in $\mathcal{M}_{p}, p \geqslant 1$ and $\alpha = p$. The following are equivalent:
\begin{enumerate}
\item $\mu\left(y,\infty\right)$ is regularly varying of index $-p$.
\item $\Im r_{\Psi_{\mu}}\left(-iy^{-1}\right)$ is is slowly varying and
\begin{equation} \label{cor1.2c}
\Re r_{\Psi_{\mu}}\left(-iy^{-1}\right) \gg y^{-1}.
\end{equation}
\end{enumerate}
If any of the above statements holds, we also have, as $z \rightarrow 0$ n.t., 
\begin{equation*} \label{cor1.2a}
z \ll r_{\Psi_{\mu}}(z);
\end{equation*}
as $y \rightarrow \infty$
\begin{equation} \label{cor1.2b}
\Im r_{\Psi_{\mu}}\left(-iy^{-1}\right) \sim -\frac{\pi}{2}y^{p}\mu\left(y,\infty\right) \gg y^{-1}
\end{equation}
\end{theorem}
If we consider $\alpha \in [0,1)$, then
\begin{theorem} \label{cor1.3}
Let $\mu$ be in $\mathcal{M}_{0}$ and $0 \leq \alpha < 1$. The following are equivalent:
\begin{enumerate}
\item $\mu\left(y,\infty\right)$ is regularly varying of index $-\alpha$.
\item $\Im \Psi_{\mu}\left(-iy^{-1}\right)$ is regularly varying of index $-\alpha$ and $$\Re \Psi_{\mu}\left(-iy^{-1}\right) \approx \Im\Psi_{\mu}\left(-iy^{-1}\right).$$
\end{enumerate}
If any of the above statements holds, we also have, as $z \rightarrow 0$ n.t., 
\begin{equation*} \label{cor1.3a}
z \ll {\Psi_{\mu}}(z);
\end{equation*}
as $y \rightarrow \infty$
\begin{align} \label{cor1.3b}
\Im \Psi_{\mu}\left(-iy^{-1}\right) &\sim -\frac{\pi\left(1-\alpha\right)/2}{\cos\left(\pi\alpha/2\right)}\mu\left(y,\infty\right) \gg y^{-1}
\end{align}
and
\begin{equation} \label{cor1.3c}
\Re \Psi_{\mu}\left(-iy^{-1}\right) \sim -d_{\alpha}\mu\left(y,\infty\right) \gg y^{-1}
\end{equation}
where $d_{\alpha}$ is as in \eqref{dalpha}.
\end{theorem}
Finally, when $p \geqslant 0$ and $\alpha = p+1$.
\begin{theorem} \label{cor1.5}
Let $\mu$ be in $\mathcal{M}_{p}, p \geqslant 1$ and $ \alpha = p+1, r \in \left(0,1/2\right)$. The following are equivalent:
\begin{enumerate}
\item $\mu\left(y,\infty\right)$ is regularly varying of index $-\left(p+1\right)$.
\item $\Re r_{\Psi_{\mu}}\left(-iy^{-1}\right)$ is regularly varying of index $-1$ and
\begin{equation} \label{cor1.5c}
y^{-1} \ll \Im r_{\Psi_{\mu}}\left(-iy^{-1}\right) \ll y^{-\left(1-r/2\right)}.
\end{equation}
\end{enumerate}
If any of the above statements holds, we also have, as $z \rightarrow 0$ n.t., 
\begin{equation*} \label{cor1.5a}
z \ll r_{\Psi_{\mu}}\left(z\right);
\end{equation*}
as $y \rightarrow \infty$
\begin{align} \label{cor1.5b}
y^{-(1+r/2)} \ll \Re r_{\Psi_{\mu}}\left(-iy^{-1}\right) &\sim -\frac{\pi}{2}y^{p}\mu\left(y,\infty\right) \ll y^{-\left(1-r/2\right)}.
\end{align}
\end{theorem}
To study the relation between the remainder terms of $\Psi$ and $\eta$ transforms we consider the following classes of functions which contains $\Psi_{\mu}$ depending on regular variation of $\mu$. We shall show that the classes are closed under certain operations. 
Let $\mathcal{H}$ denote the set of analytic functions $A$ having a domain $\mathcal{D}_A$ such that for all positive $\kappa$, there exists $\delta > 0$ with $\Delta_{\kappa, \delta} \subset \mathcal{D}_A$.
\begin{definition} \label{defZ}
Let $Z_{1,p}$ denote the set of all $A \in \mathcal{H}$ which satisfies the following conditions:

\let\myenumi\theenumi
\renewcommand{\theenumi}{R\myenumi}
\begin{enumerate}
\item For $p \geqslant 0$, $A$ has Taylor series expansion with real coefficients of the form $$A(z) = \sum_{j=1}^{p}a_jz^j + z^pr_A(z)$$ where $a_1, \cdots , a_p$ are real numbers and for $p = 0$ we interpret the term in the sum as absent.\label{R1}
\item $z \ll r_A(z) \ll 1$ as $z \to 0$ n.t.\label{R2}
\item $\Re r_A(-iy^{-1}) \approx \Im r_A(-iy^{-1})$ as $y \to \infty$.\label{R3}
\end{enumerate}

Let $Z_{2,p}$ be the same as $Z_{1,p}$ with \eqref{R1} and \eqref{R2} but \eqref{R3} is replaced by
\renewcommand{\theenumi}{R3$^\prime$}
\begin{enumerate}
\item $y^{-(1+r/2)} \ll \Re r_A(-iy^{-1}) \ll y^{-(1-r/2)}$ and $y^{-1} \ll \Im r_A(-iy^{-1}) \ll y^{-(1-r/2)}$ for any $r \in (0,1/2)$.\label{R3prime}
\end{enumerate}
Let $Z_{3,p}$ be the same as $Z_{1,p}$ with \eqref{R1} for $p \geqslant 1$ and same \eqref{R2} but \eqref{R3} is replaced by
\renewcommand{\theenumi}{R3$^{\prime\prime}$}
\begin{enumerate}
\item $\Re r_A(-iy^{-1}) \gg y^{-1}$ and $\Im r_A(-iy^{-1}) \gg y^{-1}$. \label{R3primeprime}
\end{enumerate}
\let\theenumi\myenumi
\end{definition}
\begin{remark} \label{rmkZ}
Suppose that $\mu(y,\infty)$ is regularly varying $-\alpha$ and $\mu \in \mathcal{M}_p$ with $\alpha \in [p,p+1]$. Then note the following:
\begin{enumerate}
\item If $p \geqslant 1$, $p <\alpha < p+1$ or $p=0$, $0 \leqslant \alpha <1$, then $\Psi_{\mu}(z) \in Z_{1,p}$. This follows from Theorem \ref{cor1.1} and Theorem \ref{cor1.3}.
\item If $p \geqslant 0$, $\alpha = p+1$, then $\Psi_{\mu}(z) \in Z_{2,p}$. This follows from Theorem \ref{cor1.5}.
\item If $p \geqslant 1$, $\alpha = p$, then $\Psi_{\mu}(z) \in Z_{3,p}$. This follows from Theorem \ref{cor1.2}.
\end{enumerate}
Hence the proposition \ref{etatrans}, given below, allows us to conclude that $\Psi_{\mu}(z) \in Z_{i,p}$ if and only if $\eta_{\mu}(z) \in Z_{i,p}$ for any fixed $i \in \{1,2,3\}$ and $p \in \{0,1,2, \cdots \}$.
\end{remark}

\begin{proposition} \label{etatrans}
For any fixed $i \in \{1,2,3\}$ and $p \geqslant 0$ (excluding $Z_{3,0}$ as this set is not defined), if $A(z) \in Z_{i,p}$ then $B(z) = A(z)(1 \pm A(z))^{-1} \in Z_{i,p}$. Furthermore, we have
\begin{enumerate}
\item $r_B(z) \sim r_A(z)$, as $z \to 0$ n.t.;
\item $\Re r_B(-iy^{-1}) \sim \Re r_A(-iy^{-1})$ as $y \to \infty$ and 
\item $\Im r_B(-iy^{-1}) \sim \Im r_A(-iy^{-1})$ as $y \to \infty$.
\end{enumerate}
\end{proposition}
\begin{proof}
We shall divide this proof into some cases because depending on $i$ and $p$ the calculations are different. We shall only show for $B(z) = A(z)(1 + A(z))^{-1}$. Exactly same calculation will prove the result for $B(z) = A(z)(1 - A(z))^{-1}$.
\begin{enumerate}
\item Suppose $A \in Z_{1,0}$. Then $r_A(z) = A(z)$. Therefore $r_B(z) = B(z)$. This shows that \eqref{R1} is satisfied.

Now,
\begin{align}
B(z) &= A(z)(1 + A(z))^{-1} \label{Psi111} \\ 
&= A(z) + o(|A(z)|)  \text{  as  } z \rightarrow 0 \text{ n.t.} \label{Psi11}
\end{align}  
Therefore, we have $B(z) \sim A(z)$ as $z \rightarrow 0$ n.t. Therefore \eqref{R2} is satisfied.

From ~\eqref{Psi11} we have,
\begin{align*}
\Re B(-iy^{-1}) &= \Re A(-iy^{-1}) + o(|A(-iy^{-1})|),\\
\Im B(-iy^{-1}) &= \Im A(-iy^{-1}) + o(|A(-iy^{-1})|).
\end{align*}
Now to show the equivalence of the real parts and imaginary parts, it is enough to show that 
$$
\frac{|A(-iy^{-1})|}{\Re A(-iy^{-1})} 
\text{ and } 
\frac{|A(-iy^{-1})|}{\Im A(-iy^{-1})}
$$
remains bounded as $y \rightarrow \infty$. We shall show the first part only as the second one follows by the same arguments.
\begin{align*}
\left(\left|\frac{A(-iy^{-1})}{\Re A(-iy^{-1})}\right|\right)^2  &= \frac{(\Re A(-iy^{-1}))^2 + (\Im A(-iy^{-1}))^2}{({\Re A(-iy^{-1})})^2} \\   
 &= 1+\left(\frac{\Im A(-iy^{-1})}{\Re A(-iy^{-1})}\right)^2,
\end{align*}
which goes to a constant as $y \rightarrow \infty$ by the fact that $A$ satisfies \eqref{R3}. Therefore \eqref{R3} is satisfied for $B(z)$ and the asymptotics in the statement also remain true.
\item Suppose $A \in Z_{1,p}$, $p \geqslant 1$. We note that $|A (z)| \to 0$ as $z \to 0$ n.t. Thus we have the following series expansion using equation \eqref{etapsi} for $B (z)$ near zero:
\begin{align*}
B (z) = \sum_{i=1}^{p}(-1^{i+1})(A (z))^{i} + \text{O}\left(({A (z)})^{p+1}\right).
\end{align*}
Using \eqref{R1} and \eqref{R2} we get $$\frac{({A (z)})^{p+1}}{z^{p}{r_{A }(z)}} =  \left(\frac{A (z)}{z}\right)^{p+1}\frac{z}{r_{A }(z)}\to 0$$ as $z \to 0$. Hence
\begin{align}\label{etamun10}
B (z) = \sum_{i=1}^{p}((-1^{i+1})(\sum_{j=1}^{p}m_{j}z^{j} + z^{p}r_{A }(z))^{j}) + \text{o}\left(z^{p}{r_{A }(z)}\right).
\end{align}
We expand the term in the right-hand side of \eqref{etamun100}. As $z \ll r_{A }(z)$, all powers of $z$ with indices greater than $p$ can be absorbed in the last term on the right-hand side. Then collect upto $p$-th power of $z$ to form a polynomial $P(z)$ of degree at most $p$ with real coefficients without the constant term. Finally we consider the terms containing some powers of $r_{A }(z)$ which will contain terms of the form $z^{l_{1}}(z^{p}r_{A }(z))^{l_{2}}$ for integers $l_{1} \geqslant 0$ and $l_{2} \geqslant 1$ with leading term $z^{p}r_{A }(z)$ and the remaining terms can be absorbed in the last term in the right-hand side. Thus we get,
\begin{align*}
B (z) = P(z) + z^{p}r_{A }(z) + \text{o}(z^{p}r_{A }(z)).
\end{align*}
Therefore \eqref{R1} is satisfied. Now by uniqueness of the Taylor series expansion, we have
\begin{align}\label{etamun}
r_{B }(z) = r_{A }(z) + \text{o}(r_{A }(z)).
\end{align}
 Therefore $r_{B }\left(z\right) \sim r_{A }\left(z\right)$. So \eqref{R2} is satisfied. Thus
\begin{align*}
\Re r_{B }\left(z\right) &= \Re r_{A }\left(z\right) + o\left(|r_{A }\left(z\right)|\right), \\
\Im r_{B }\left(z\right) &= \Im r_{A }\left(z\right) + o\left(|r_{A }\left(z\right)|\right).
\end{align*}
Now from \eqref{R3} and same calculations like in the first case we get that $B(z)$ is satisfying \eqref{R3}.
\item Suppose $A \in Z_{2,0}$. Here we only need to show \eqref{R3prime} as \eqref{R1} and \eqref{R2} have been already shown in case $1$. From \eqref{Psi111},
\begin{align*}
B \left(z\right) &= A \left(z\right) + O\left(|A \left(z\right)|^2\right).
\end{align*}
Consequently,
\begin{align*}
\Re B \left(z\right) &= \Re A \left(z\right) + O\left(|A \left(z\right)|^2\right),\\
\Im B \left(z\right) &= \Im A \left(z\right) + O\left(|A \left(z\right)|^2\right).
\end{align*}
It is enough to show that 
$\frac{|A \left(-iy^{-1}\right)|^2}{\Re A \left(-iy^{-1}\right)}$ and $\frac{|A \left(-iy^{-1}\right)|^2}{\Im A \left(-iy^{-1}\right)}$ both goes to zero as $y \rightarrow \infty$. For that
\begin{align*}
\frac{|A \left(-iy^{-1}\right)|^2}{\Re A \left(-iy^{-1}\right)} &= \Re A \left(-iy^{-1}\right) + \frac{\left(\Im A \left(-iy^{-1}\right)\right)^2}{\Re A \left(-iy^{-1}\right)}\\ &= \Re A \left(-iy^{-1}\right)+\left(\frac{\Im A \left(-iy^{-1}\right)}{y^{-\left(1-r/2\right)}}\right)^2\frac{y^{-\left(1+r/2\right)}}{\Re A \left(-iy^{-1}\right)}y^{-1+3r/2},
\end{align*}
which goes to zero as $y \rightarrow \infty$ using \eqref{R3prime} for $A(z)$. For the other terms similarly note that
\begin{align*}
\frac{|A \left(-iy^{-1}\right)|^2}{\Im A \left(-iy^{-1}\right)} &= \Im A \left(-iy^{-1}\right) + \frac{\left(\Re A \left(-iy^{-1}\right)\right)^2}{\Im A \left(-iy^{-1}\right)} \\ &= \Im A \left(-iy^{-1}\right) + \left(\frac{\Re A \left(-iy^{-1}\right)}{y^{-\left(1-r/2\right)}}\right)^2\frac{y^{-1}}{\Im A \left(-iy^{-1}\right)}y^{-1+r}
\end{align*}
also goes to zero as $y \rightarrow \infty$ using \eqref{R3prime} for $A(z)$. Thus \eqref{R3prime} is obviously satisfied by $B(z)$.
\item Suppose $A \in Z_{2,p}$, $p \geqslant 1$. Here we need to show only \eqref{R3prime}. From \eqref{etamun10} we can write using a similar argument given in case (2),
\begin{align*}
B (z) = P(z) + z^{p}r_{A }(z) + c_{1}z^{p+1} + \text{O}(z^{p+1}r_{A }(z)).
\end{align*}
Therefore,
\begin{align}\label{imagine}
r_{B }\left(z\right) &= r_{A }(z) + c_{1}z + \text{O}(zr_{A }(z)).
\end{align}
So, $r_{B }\left(z\right) \sim r_{A }\left(z\right)$ and evaluating \eqref{imagine} at the point $z=-iy^{-1}$,
\begin{align*}
r_{B }\left(-iy^{-1}\right) &= r_{A }(-iy^{-1}) + c_{1}(-iy^{-1}) + \text{O}(zr_{A }(-iy^{-1}))
\end{align*}
and after taking the real parts on both sides, we have
\begin{align*}
\Re r_{B }\left(-iy^{-1}\right) &= \Re r_{A }\left(-iy^{-1}\right) + O\left(y^{-1}|r_{A }\left(-iy^{-1}\right)|\right).
\end{align*}
Now,
\begin{align*}
\left|\frac{y^{-1}|r_{A }\left(-iy^{-1}\right)|}{\Re r_{A }\left(-iy^{-1}\right)}\right|^2 &= \frac{1}{y^2} + \frac{1}{y^2}\left(\frac{\Im r_{A }\left(-iy^{-1}\right)}{\Re r_{A }\left(-iy^{-1}\right)}\right)^2
\end{align*}
and
\begin{align*}
y^{-1}\frac{\Im r_{A }\left(-iy^{-1}\right)}{\Re r_{A }\left(-iy^{-1}\right)} &= \frac{\Im r_{A }\left(-iy^{-1}\right)}{y^{-\left(1-r/2\right)}}\frac{y^{-\left(1+r/2\right)}}{\Re r_{A }\left(-iy^{-1}\right)}y^{-1+r}
\end{align*}
goes to zero as $y \rightarrow \infty $ by \eqref{R3prime}. As a consequence we conclude that $\Re r_{B }\left(-iy^{-1}\right) \sim \Re r_{A }\left(-iy^{-1}\right)$.

For the imaginary part asymptotic we write from \eqref{imagine},
\begin{align}\label{etamun100}
r_{B }\left(z\right) &= r_{A }(z) + \text{O}(|z|).
\end{align}
Now putting $z = -iy^{-1}$ and taking imaginary parts on both sides we get 
\begin{align*}
\Im r_{B }\left(-iy^{-1}\right) &= \Im r_{A }(-iy^{-1}) + \text{O}(|y^{-1}|)
\end{align*}
 and note that ${y\Im r_{A }\left(-iy^{-1}\right)} \rightarrow \infty$ as $z \rightarrow \infty$ by \eqref{R3prime}. Hence $\Im r_{B }\left(-iy^{-1}\right) \sim \Im r_{A }(-iy^{-1})$. So we are done in this case.
 \item Suppose $A \in Z_{3,p}$, $p \geqslant 1$. Here \eqref{R1} and \eqref{R2} is shown in case (2). Now we write the following using \eqref{etamun100},
\begin{align*}
\Re r_{B }\left(-iy^{-1}\right) &= \Re r_{A }\left(-iy^{-1}\right) + O\left(|y^{-1}|\right),\\
\Im r_{B }\left(-iy^{-1}\right) &= \Im r_{A }\left(-iy^{-1}\right) + O\left(|y^{-1}|\right).
\end{align*}
Now ${y\Re r_{A }\left(-iy^{-1}\right)}$ and ${y\Im r_{A }\left(-iy^{-1}\right)}$ both go to infinity as $y \rightarrow \infty$ by \eqref{R3primeprime}. This shows that \eqref{R3primeprime} is also satisfied by $B(z)$ in this case.
\end{enumerate}
\end{proof}
We give the proofs of the main theorems stated in Theorem~\ref{ccor1.1}--Theorem~\ref{ccor1.5}. We shall only prove Theorem~\ref{ccor1.1} and the rest of the theorems will follow by similar arguments.

\textbf{Proof of Theorem~\ref{ccor1.1}}: Combining Theorem \ref{cor1.1}, Proposition~\ref{etatrans} and the definition of $B$-transform coming out of $\eta$-transform, we get Theorem\ref{ccor1.1} because the asymptotic relationship of $\Psi$-transform follows from Theorem~\ref{cor1.1}, the relationship between $\Psi$ and $\eta$ transforms follows from the Proposition ~\ref{etatrans} (see also Definition \ref{defZ} and Remark \ref{rmkZ}) and finally the correspondence between $\eta$ and $B$ transforms is ensured by equations stated in \eqref{eq:t}. \qed

\textbf{Proof of Theorem \ref{Bmuz}:} Recall that we have assumed $\mu \in \mathcal{M}_p$ is regularly varying tail index $-\alpha$, $\alpha \geqslant 0$ with $\alpha \in [p,p+1]$. Since $\nu$ is regularly varying with tail index $-\beta$ with $\beta \geqslant \alpha$ and $\nu$ has finite first moment, there exists a  positive integer $q \geqslant 1$ with $q \geqslant p$ and $\nu \in \mathcal{M}_{q}$. The proof of this theorem is split up into different cases depending on $p$, $q$ and $\alpha$, $\beta$. It contains a number of subcases because the asymptotic relations and the Taylor series like expansions differ with respect to the position of $\alpha$ and $\beta$ in the lattice of non-negative integers.

{\bf Case \eqref{theoremoa}}\label{wone} Recall that in this case we have assumed either $1 \leqslant p < \alpha < p+1$ or $0=p \leqslant \alpha < p+1=1$.

{\bf Subcase (i) } Let $\mu \in \mathcal{M}_{0}, \nu \in \mathcal{M}_{q}$, $2 \leqslant q$ and $0 \leqslant \alpha < 1$ and $q \leqslant \beta \leqslant q+1$. Define 
\begin{align*}
f_{1}\left(z\right) := \frac{1}{B_{\mu}}\left(z\right) \text{ and }
f_{2}\left(z\right) := \frac{1}{m(\nu)B_{\nu}}\left(z\right) &= 1 + d_{1}z + \cdots + d_{q-1}z^{q-1} + z^{q-1}r_{f_{2}}\left(z\right),
\end{align*}
where $d_j$, $j=1,2, \cdots, q-1$ are real coefficients. Therefore
\begin{align*}
f_{1}\left(z\right)f_{2}\left(z\right) &= f_{1}\left(z\right) + d_{1}zf_{1}\left(z\right) + \cdots + d_{q-1}z^{q-1}f_{1}\left(z\right) + z^{q-1}f_{1}\left(z\right)r_{f_{2}}\left(z\right)
\end{align*}
and using Lemma \ref{mnbowtie}
\begin{align} \label{rf1f2}
r_{f_{1}f_{2}}(z) = f_{1}\left(z\right)f_{2}\left(z\right) = f_{1}\left(z\right) + O\left(|zf_{1}\left(z\right)|\right).
\end{align}
Taking the imaginary parts of \eqref{rf1f2}, we get $$\Im r_{f_{1}f_{2}}(-iy^{-1}) = \Im f_{1}\left(-iy^{-1}\right) + O\left(y^{-1}|f_{1}\left(-iy^{-1}\right)|\right).$$ 
Observe,
\begin{align*}
y^{-1}\frac{|f_{1}\left(-iy^{-1}\right)|}{\Im f_{1}\left(-iy^{-1}\right)} &= \left(\frac{\left(\Re f_{1}\left(-iy^{-1}\right)\right)^2+\left(\Im f_{1}\left(-iy^{-1}\right)\right)^2}{y^2\left(\Im f_{1}\left(-iy^{-1}\right)\right)^2}\right)^{1/2}\\ &= \left(\frac{1}{y^2} + \frac{1}{y^2}\left(\frac{\Re f_{1}\left(-iy^{-1}\right)}{\Im f_{1}\left(-iy^{-1}\right)}\right)^2\right)^{1/2} \rightarrow 0 \text{  as } y\rightarrow \infty,  
\end{align*}
using ~\eqref{ccor1.3b} and \eqref{ccor1.3c}. Hence $\Im r_{f_{1}f_{2}}(-iy^{-1}) \sim \Im f_{1}\left(-iy^{-1}\right)$ as $y \to \infty$. Using same arguments for real part we also have $\Re r_{f_{1}f_{2}}(-iy^{-1}) \sim \Re f_{1}\left(-iy^{-1}\right)$.

{\bf Subcase (ii)} Let $\mu \in \mathcal{M}_{0}, \nu \in \mathcal{M}_{q}$, $q =1$ and $0 \leqslant \alpha < 1$ and $q \leqslant \beta \leqslant q+1$. The case is similar to $q \geqslant 2$. Here the equation \eqref{rf1f2} gets replaced by   
\begin{align*}
r_{f_{1}f_{2}}(-iy^{-1}) = f_{1}\left(-iy^{-1}\right) + f_{1}\left(-iy^{-1}\right)r_{f_{2}}\left(-iy^{-1}\right).
\end{align*}
Taking real part and imaginary part we have $\Re r_{f_{1}f_{2}}(-iy^{-1}) \sim \Re f_{1}\left(-iy^{-1}\right)$ and $\Im r_{f_{1}f_{2}}(-iy^{-1}) \sim \Im f_{1}\left(-iy^{-1}\right)$  respectively using $\frac{\Im f_{1}\left(-iy^{-1}\right)}{\Im f_{1}\left(-iy^{-1}\right)}$, $\frac{\Re f_{1}\left(-iy^{-1}\right)}{\Re f_{1}\left(-iy^{-1}\right)}$, $\frac{\Re f_{1}\left(-iy^{-1}\right)}{\Im f_{1}\left(-iy^{-1}\right)} \sim \text{a nonzero constant}$ and $|r_{f_{2}}\left(-iy^{-1}\right)| \to 0$ both as $y \rightarrow \infty$. So the $p = 0$ case is done.

{\bf Subcase (iii)} Here suppose $\mu, \nu \in \mathcal{M}_{1}$ and $1 < \alpha < \beta < 2$. Let $f_{1}\left(z\right) = \frac{1}{m(\mu)B_{\mu}}\left(z\right)$ and $f_{2}\left(z\right) = \frac{1}{m(\nu)B_{\nu}}\left(z\right)$. Then
\begin{align*}
f_{i}\left(z\right) &= 1 + r_{f_{i}}\left(z\right) \text{  for  } i=1,2
\end{align*}
Therefore,
\begin{align*}
r_{f_{1}f_{2}}(z) &= r_{f_{1}}\left(z\right) + r_{f_{2}}\left(z\right) +r_{f_{1}}\left(z\right)r_{f_{2}}\left(z\right).
\end{align*}
So,
\begin{align}\label{needeq1}
 r_{f_{1}f_{2}}(-iy^{-1}) &= r_{f_{1}}\left(-iy^{-1}\right) + r_{f_{2}}\left(-iy^{-1}\right) +r_{f_{1}}\left(-iy^{-1}\right)r_{f_{2}}\left(-iy^{-1}\right)\\
 &=  r_{f_{1}}\left(-iy^{-1}\right) +O\left(|r_{f_{2}}\left(-iy^{-1}\right)|\right) \nonumber.
\end{align}
Thus taking the real and imaginary parts $$\Im r_{f_{1}f_{2}}(-iy^{-1}) = \Im r_{f_{1}}\left(-iy^{-1}\right) +O\left(|r_{f_{2}}\left(-iy^{-1}\right)|\right)$$  and  $$\Re r_{f_{1}f_{2}}(-iy^{-1}) = \Re r_{f_{1}}\left(-iy^{-1}\right) +O\left(|r_{f_{2}}\left(-iy^{-1}\right)|\right).$$
Using ~\eqref{ccor1.1b} and ~\eqref{ccor1.1c},
\begin{align*}
\left(\frac{|r_{f_{2}}\left(-iy^{-1}\right)|}{\Im r_{f_{1}}\left(-iy^{-1}\right)}\right)^2 &= \frac{\left(\Re r_{f_{2}}\left(-iy^{-1}\right)\right)^2 + \left(\Im r_{f_{2}}\left(-iy^{-1}\right)\right)^2}{\left(\Im r_{f_{1}}\left(-iy^{-1}\right)\right)^2} \sim \frac{y^{1-\beta}l_{1}(y)}{y^{1-\alpha}l_{2}(y)} \underset{y \to \infty}\to 0,
\end{align*}
where $l_{k}(y)$, $k=1,2$ are slowly varying functions. Therefore we get $\Im r_{f_{1}f_{2}}(-iy^{-1}) \sim \Im r_{f_{1}}\left(-iy^{-1}\right)$. Similarly taking the real parts one can show $\Re r_{f_{1}f_{2}}(-iy^{-1}) \sim \Re r_{f_{1}}\left(-iy^{-1}\right)$.

{\bf Subcase (iv)} Let $\mu, \nu \in \mathcal{M}_{1}$ and $1 < \alpha < 2, \beta=2$. The case is similar to the previous one. Using \eqref{needeq1}, we get 
\begin{align*}
\Im r_{f_{1}f_{2}}(z) &= \Im r_{f_{1}}\left(z\right) + \Im r_{f_{2}}\left(z\right) + \Im \left(r_{f_{1}}\left(z\right)r_{f_{2}}\left(z\right)\right).
\end{align*}
So the expression for the imaginary part becomes
\begin{align*}
\Im r_{f_{1}f_{2}}(z) &= \Im r_{f_{1}}\left(z\right) + \Im r_{f_{2}}\left(z\right) + \Im r_{f_{1}}\left(z\right) \Re r_{f_{2}}\left(z\right) + \Re r_{f_{1}}\left(z\right) \Im r_{f_{2}}\left(z\right).
\end{align*}
Using similar type of arguments with the help of ~\eqref{ccor1.5b} and ~\eqref{ccor1.5c}, we get 
$\Im r_{f_{1}f_{2}}(-iy^{-1}) \sim \Im r_{f_{1}}\left(-iy^{-1}\right)$. The real part can be dealt similarly.

{\bf Subcase (v)} Suppose $\mu, \nu \in \mathcal{M}_{p}, p=q \geqslant 2$ and $p < \alpha < \beta \leqslant p+1$. Write 
\begin{align*}
f_{1}\left(z\right) := \frac{1}{m(\mu)B_{\mu}}\left(z\right) &= 1 + c_{1}z + \cdots + c_{p-1}z^{p-1} + z^{p-1}r_{f_{1}}\left(z\right).\\
f_{2}\left(z\right) := \frac{1}{m(\nu)B_{\nu}}\left(z\right) &= 1 + d_{1}z + \cdots + d_{p-1}z^{p-1} + z^{p-1}r_{f_{2}}\left(z\right),
\end{align*}
where $c_{i}$ and $d_{i}$, $1\leqslant i \leqslant p-1$ are some real constants. So,
\begin{align*}
f_{1}\left(z\right)f_{2}\left(z\right) &= 1 + e_{1} z + \cdots + e_{p-1}z^{p-1} + z^{p-1}\left(r_{f_{1}}\left(z\right) + r_{f_{2}}\left(z\right) +O\left(|z|\right)\right),
\end{align*}
where $e_{i}, 1\leqslant i \leqslant p-1$ are real constants. Therefore,
\begin{align*} 
r_{f_{1}f_{2}}(z) = r_{f_{1}}\left(z\right) + r_{f_{2}}\left(z\right) +O\left(|z|\right).
\end{align*}
Taking imaginary and real part on both sides we obtain 
\begin{align}
\Im r_{f_{1}f_{2}}(-iy^{-1}) &= \Im r_{f_{1}}\left(-iy^{-1}\right) + \Im r_{f_{2}}\left(-iy^{-1}\right) + O\left(y^{-1}\right), \label{needeq2} \\
\Re r_{f_{1}f_{2}}(-iy^{-1}) &= \Re r_{f_{1}}\left(-iy^{-1}\right) + \Re r_{f_{2}}\left(-iy^{-1}\right) + O\left(y^{-1}\right). \label{needreal2}
\end{align}
When $p < \alpha < \beta < p+1$, we can use \eqref{ccor1.1b} and \eqref{ccor1.1c} to get  
$$\frac{\Im r_{f_{2}}\left(-iy^{-1}\right)}{\Im r_{f_{1}}\left(-iy^{-1}\right)} \text{ and } \frac{\Re r_{f_{2}}\left(-iy^{-1}\right)}{\Re r_{f_{1}}\left(-iy^{-1}\right)} 
\rightarrow 0 \text{ as }  y\rightarrow \infty.$$ Also, ${y\Im r_{f_{1}}\left(-iy^{-1}\right)}, {y\Re r_{f_{1}}\left(-iy^{-1}\right)} \rightarrow \infty$ as $y\rightarrow \infty$ by ~\eqref{ccor1.1b} and ~\eqref{ccor1.1c} in the respective cases.

When $p < \alpha < p+1, \beta = p+1 $ we have $$\frac{\Im r_{f_{2}}\left(-iy^{-1}\right)}{\Im r_{f_{1}}\left(-iy^{-1}\right)} = -\frac{\Im r_{f_{2}}\left(-iy^{-1}\right)}{y^{-\left(1-r/2\right)}}\frac{1}{cy^{\left(p-\alpha\right)}l\left(y\right)}\frac{1}{y^{1-r/2}},$$
where $c$ is a constant and $l\left(y\right)$ is a slowly varying function and $r$ as in Theorem \ref{ccor1.5}. To make this quantity tend to zero as $y\rightarrow \infty$ using ~\eqref{ccor1.1b} and ~\eqref{ccor1.5b} we need $\left(p-\alpha+1-r/2\right) > 0$. This can be done by a suitable choice of $r \in \left(0,1/2\right)$ since $p+1-\alpha >0$. Exactly same can be done for the real parts also.

{\bf Subcase (vi)} Now suppose $\mu \in \mathcal{M}_{p}, \nu \in \mathcal{M}_{q}, 1 \leqslant p < q$ and $p < \alpha < p+1$ and $q \leqslant \beta \leqslant q+1$. Here we have,
\begin{align}\label{prf1f21}
f_{1}\left(z\right) := \frac{1}{m(\mu)B_{\mu}}\left(z\right) &= 1 + c_{1}z + \cdots + c_{p-1}z^{p-1} + z^{p-1}r_{f_{1}}\left(z\right)
\end{align}
and
\begin{align}\label{prf1f22}
f_{2}\left(z\right) := \frac{1}{m(\nu)B_{\nu}}\left(z\right) = 1 + d_{1}z + \cdots +d_{p-1}z^{p-1}+
 \cdots + d_{q-1}z^{q-1} + z^{q-1}r_{f_{2}}\left(z\right),
\end{align}
where $c_{i}, 1\leqslant i \leqslant p-1$ and $d_{j}, 1\leqslant j \leqslant q-1$ are some real constants. 
It is easy to see using $p < q$ we have
\begin{equation}\label{case4eqnneed}
 r_{f_{1}f_{2}}(z) = r_{f_{1}}\left(z\right) + O\left(|z|\right).
\end{equation} 
Observe that the asymptotics follow since ${y\Im r_{f_{1}}\left(-iy^{-1}\right)}, {y\Re r_{f_{1}}\left(-iy^{-1}\right)} \rightarrow \infty$ as $y\rightarrow \infty$ (using ~\eqref{ccor1.1b} and \eqref{ccor1.1c} respectively).
\vskip10pt

{\bf Case \eqref{theoremob}} The second part of the theorem deals with the case $p \geqslant 1$ and $\alpha = p$. We split again the proof into several subcases.

{\bf Subcase (i)} $1 \leqslant p=\alpha < \beta < p+1$ or $ \alpha =p$,  $p <q \leqslant \beta \leqslant q+1$. In this case we can define $f_1$ and $f_2$ as in \eqref{prf1f21} and \eqref{prf1f22} respectively and one can get the imaginary part asymptotics using the similar calculations as in the proof of  \eqref{theoremoa}. We only need to show  $\Re r_{f_{1}f_{2}}(-iy^{-1}) \gg y^{-1}$ in the above cases. But it is obvious since $p,q \geqslant 1$ and $\Im r_{f_{1}\left(-iy^{-1}\right)}, \Re r_{f_{1}\left(-iy^{-1}\right)} \gg y^{-1}$.

{\bf Subcase (ii)} Suppose $p=q=1$, $\alpha = p$ and $\beta = p+1$.  Then we have form \eqref{needeq1}
\begin{align}
 \Im r_{f_{1}f_{2}}(-iy^{-1}) &= \Im r_{f_{1}}\left(-iy^{-1}\right) + \Im r_{f_{2}}\left(-iy^{-1}\right) +\Im r_{f_{1}}\left(-iy^{-1}\right) \Re r_{f_{2}}\left(-iy^{-1}\right) \nonumber \\ & + \Re r_{f_{1}}\left(-iy^{-1}\right) \Im r_{f_{2}}\left(-iy^{-1}\right) \nonumber\\
 \Re r_{f_{1}f_{2}}(-iy^{-1}) &= \Re r_{f_{1}}\left(-iy^{-1}\right) + O( |r_{f_{2}}\left(-iy^{-1}\right)|)  \label{realparteq1}.
\end{align}
Now using \eqref{ccor1.5c}, \eqref{ccor1.2b}, we get 
\begin{align} \label{ekjaygay}
\frac{\Im r_{f_{2}}\left(-iy^{-1}\right)}{\Im r_{f_{1}}\left(-iy^{-1}\right) } = \frac{\Im r_{f_{2}}\left(-iy^{-1}\right) }{y^{-(1-r/2)}}\frac{1}{y^{1-r/2}}\frac{1}{l(y)} \to 0 \text{  as  } y \to \infty,
\end{align}
also observe $$\frac{\Im r_{f_{1}}\left(-iy^{-1}\right) \Re r_{f_{2}}\left(-iy^{-1}\right)}{\Im r_{f_{1}}\left(-iy^{-1}\right)} = \Re r_{f_{2}}\left(-iy^{-1}\right) \to 0 \text{  as  } y \to \infty.$$ 
Therefore $\Im r_{f_{1}f_{2}}(-iy^{-1}) \sim \Im r_{f_{1}}\left(-iy^{-1}\right)$.  We now consider the equation \eqref{realparteq1}. Observe that from \eqref{ccor1.2c} and \eqref{ccor1.5a} $$\Re r_{f_{1}}\left(-iy^{-1}\right)  \gg y^{-1} \text{ and } y\left|r_{f_{2}}\left(-iy^{-1}\right)\right| = \left|\frac{r_{f_{2}}\left(-iy^{-1}\right)}{-iy^{-1}}\right| \gg 1.$$ These show that $\Re r_{f_{1}f_{2}}(-iy^{-1})  \gg y^{-1}$.

{\bf Subcase (iii)} Now let $p = q \geqslant 2$, $\alpha =p$ and $\beta = p+1$.  For imaginary part again using  the equations \eqref{needeq2}, \eqref{ccor1.2b} and calculations as in \eqref{ekjaygay} we have $\Im r_{f_{1}f_{2}}(-iy^{-1}) \sim \Im r_{f_{1}}\left(-iy^{-1}\right)$. The real part asymptotics follow from  \eqref{ccor1.2c}, \eqref{ccor1.5b}.

{\bf Case \eqref{theoremoc}} \label{casee3} Let $\mu \in \mathcal{M}_{0}, \nu \in \mathcal{M}_{q}, 1\leqslant q$ and $\alpha = 1$ and $q \leqslant \beta \leqslant q+1$ with $\alpha \neq \beta$. In this case we have from~\eqref{rf1f2},
$\Im r_{f_{1}f_{2}}(-iy^{-1}) = \Im f_{1}\left(-iy^{-1}\right) + O\left(y^{-1}|f_{1}\left(-iy^{-1}\right)|\right)$.

Now
\begin{align*}
y^{-1}\frac{|f_{1}\left(-iy^{-1}\right)|}{\Im f_{1}\left(-iy^{-1}\right)} &= \left(\frac{\left(\Re f_{1}\left(-iy^{-1}\right)\right)^2+\left(\Im f_{1}\left(-iy^{-1}\right)\right)^2}{y^2\left(\Im f_{1}\left(-iy^{-1}\right)\right)^2}\right)^{1/2}\\ &= \left(\frac{1}{y^2} + \frac{1}{y^2}\left(\frac{\Re f_{1}\left(-iy^{-1}\right)}{\Im f_{1}\left(-iy^{-1}\right)}\right)^2\right)^{1/2} \rightarrow 0 \text{  as } y\rightarrow \infty  
\end{align*}
because
\begin{align*}
y^{-1}\frac{\Re f_{1}\left(-iy^{-1}\right)}{\Im f_{1}\left(-iy^{-1}\right)} &= y^{-1}\frac{\Re f_{1}\left(-iy^{-1}\right)}{y^{r/2}}\frac{y^{r/2}}{\Im f_{1}\left(-iy^{-1}\right)} \\
&= \frac{\Re f_{1}\left(-iy^{-1}\right)}{y^{r/2}}\frac{1}{y^{1-r/2}\Im f_{1}\left(-iy^{-1}\right)} \rightarrow 0 \text{ as } y \rightarrow \infty
\end{align*}
using ~\eqref{ccor1.4b} and ~\eqref{ccor1.4c}. Therefore
 $\Im r_{f_{1}f_{2}}(-iy^{-1}) \sim \Im f_{1}\left(-iy^{-1}\right)$.

Again from \eqref{rf1f2},
$\Re r_{f_{1}f_{2}}(-iy^{-1}) = \Re f_{1}\left(-iy^{-1}\right) + O\left(y^{-1}|f_{1}\left(-iy^{-1}\right)|\right)$. Now if we are able to show that  $\Re r_{f_{1}f_{2}}(-iy^{-1}) \sim \Re f_{1}\left(-iy^{-1}\right)$ then we are done. For that we proceed exactly as in the case of imaginary part of this case. It is enough to show that $ y^{-1}\frac{\Im f_{1}\left(-iy^{-1}\right)}{\Re f_{1}\left(-iy^{-1}\right)} \to 0$ as $y \to \infty$ to get the required result. Note in this case  $\Re f_{1}\left(-iy^{-1}\right) \to \infty$ and using \eqref{ccor1.4b} and \eqref{ccor1.4c}, we get 
\begin{align*}
y^{-1}\frac{\Im f_{1}\left(-iy^{-1}\right)}{\Re f_{1}\left(-iy^{-1}\right)} &= y^{-1+r/2}\frac{\Im f_{1}\left(-iy^{-1}\right)}{y^{r/2}}\frac{1}{\Re f_{1}\left(-iy^{-1}\right)} 
\rightarrow 0 \text{ as } y \rightarrow \infty.
\end{align*}

{\bf Case \eqref{theoremod} }\label{casee4} Suppose $\mu \in \mathcal{M}_{p}, \nu \in \mathcal{M}_{q}, 1 \leqslant p < q$ and $\alpha = p+1$ and $q \leqslant \beta \leqslant q+1$ with $\alpha \neq \beta$. In this case we have $\Im r_{f_{1}f_{2}}(-iy^{-1}) = \Im r_{f_{1}}\left(-iy^{-1}\right)  + O\left(y^{-1}\right)$ from \eqref{case4eqnneed}. $y \Im r_{f_{1}}\left(-iy^{-1}\right) \to \infty$ as $y \to \infty$ by \eqref{ccor1.5c}.  Therefore $\Im r_{f_{1}f_{2}}(-iy^{-1}) \sim \Im r_{f_{1}}\left(-iy^{-1}\right) $.

For the real part calculations we recall \eqref{prf1f21}, \eqref{prf1f22} and write the remainder term of $f_{1}f_{2}(z)$ in the following way:
\begin{align*}
r_{f_{1}f_{2}}(z) = r_{f_{1}}\left(z\right) + d z + zr_{f_{1}}\left(z\right)+ zr_{f_{2}}\left(z\right) + O\left(|z^2|\right).
\end{align*}
Where $d$ is some real constant. We note that the term $zr_{f_{2}}\left(z\right)$ may or may not occur in the above expression depending on the value of $q-p=1$ or $>1$. Therefore
\begin{align*}
\Re r_{f_{1}f_{2}}(-iy^{-1}) &= \Re r_{f_{1}}\left(-iy^{-1}\right) - \frac{\Im r_{f_{1}}\left(-iy^{-1}\right) + \Im r_{f_{2}}\left(-iy^{-1}\right)}{y} + O\left(y^{-2}\right) .
\end{align*}

Using \eqref{ccor1.5a}, ${y^2\Re r_{f_{1}}\left(-iy^{-1}\right)} \to \infty$ as $ y\rightarrow \infty$. For the term in the middle after dividing by $\Re r_{f_{1}}(-iy^{-1})$ we observe the following:

When $q \leqslant \beta < q+1$ the numerator is regularly varying with tail index between $(-1,0]$ while the denominator is slowly varying and this allows us to conclude that the term $$\frac{\Im r_{f_{1}}\left(-iy^{-1}\right) + \Im r_{f_{2}}\left(-iy^{-1}\right)}{y\Re r_{f_{1}}\left(-iy^{-1}\right)} \to 0 \text{ as } y \to \infty.$$ Also when $\beta = q+1$ we can write 
$$\frac{\Im r_{f_{1}}\left(-iy^{-1}\right) + \Im r_{f_{2}}\left(-iy^{-1}\right)}{y\Re r_{f_{1}}\left(-iy^{-1}\right)} = \frac{1}{y^{1-r}}\frac{\Im r_{f_{1}}\left(-iy^{-1}\right) + \Im r_{f_{2}}\left(-iy^{-1}\right)}{y^{-(1-r/2)}}\frac{y^{-(1+r/2)}}{\Re r_{f_{1}}\left(-iy^{-1}\right)} \to 0 \text{ as } y \to \infty$$using \eqref{ccor1.5b} and \eqref{ccor1.5c}. Thus $\Re r_{f_{1}f_{2}}(-iy^{-1}) \sim \Re r_{f_{1}}\left(-iy^{-1}\right)$.
\vskip10pt

{\bf Proof of Theorem~\ref{Bmuz} \eqref{theoremoe}} In this case we assume $\alpha = \beta$, $\mu, \nu \in \mathcal{M}_p$ and $\nu(x,\infty) \sim c\mu(x,\infty)$ for some $c \in (0,\infty)$. The methods are similar to the previous one, so we will only briefly sketch the proofs.

{\bf Case (i) } Suppose $p \geqslant 1$ and $p < \alpha = \beta < p+1$.

{\bf Subcase (i)} First suppose $p=1$. Then from \eqref{needeq1}, we have
\begin{align*}
 r_{f_{1}f_{2}}(-iy^{-1}) &= r_{f_{1}}\left(-iy^{-1}\right) + r_{f_{2}}\left(-iy^{-1}\right) +o\left(|r_{f_{1}}\left(-iy^{-1}\right)|\right).
\end{align*}
Taking imaginary parts on both sides we get
\begin{align*}
 \Im r_{f_{1}f_{2}}(-iy^{-1}) &= \Im r_{f_{1}}\left(-iy^{-1}\right) + \Im r_{f_{2}}\left(-iy^{-1}\right) +o\left(|r_{f_{1}}\left(-iy^{-1}\right)|\right).
\end{align*}
Using \eqref{ccor1.1b}, \eqref{ccor1.1c} and the tail equivalence condition we derive
$$\frac{\Im r_{f_{2}}\left(-iy^{-1}\right)}{\Im r_{f_{1}}\left(-iy^{-1}\right)} \to c\text{ and } \left(\frac{|r_{f_{1}}\left(-iy^{-1}\right)|}{\Im r_{f_{1}}\left(-iy^{-1}\right)}\right) \to (1+c^2)^{1/2}\text{  as  } y \to \infty.$$
Therefore we have $\Im r_{f_{1}f_{2}}(-iy^{-1}) \sim (1+c)\Im r_{f_{1}}\left(-iy^{-1}\right)$. Exactly same calculations taking the real part into consideration gives $\Re r_{f_{1}f_{2}}(-iy^{-1}) \sim (1+c)\Re r_{f_{1}}\left(-iy^{-1}\right).$

{\bf Subcase (ii)} Suppose $p \geqslant 2$. Then we have the equation \eqref{needeq2}, given by
\begin{align*} 
\Im r_{f_{1}f_{2}}(-iy^{-1}) = \Im r_{f_{1}}\left(-iy^{-1}\right) + \Im r_{f_{2}}\left(-iy^{-1}\right) + O\left(y^{-1}\right).
\end{align*}
Here also using ~\eqref{ccor1.1b} we get as $y \to \infty$
$$\frac{\Im r_{f_{2}}\left(-iy^{-1}\right)}{\Im r_{f_{1}}\left(-iy^{-1}\right)} \to c \text{  
and }  y\Im r_{f_{1}}\left(-iy^{-1}\right) \to \infty.$$
Thus $\Im r_{f_{1}f_{2}}(-iy^{-1}) \sim (1+c)\Im r_{f_{1}}\left(-iy^{-1}\right)$
and exactly same calculation with real parts give  $\Re r_{f_{1}f_{2}}(-iy^{-1}) \sim (1+c)\Re r_{f_{1}}\left(-iy^{-1}\right)$.

{\bf Case (ii)} When $\alpha = p$, similar calculations like above provides $\Im r_{f_{1}f_{2}}(-iy^{-1}) \sim (1+c)\Im r_{f_{1}}\left(-iy^{-1}\right)$ since we have the same imaginary part asymptotics in this case also. The calculations done in the proof of \eqref{theoremob} assures us $\Re r_{f_{1}f_{2}}(-iy^{-1}) \gg y^{-1} $.

{\bf Case (iii)}. Here we suppose  $p \geqslant 1$, $\alpha = \beta = p+1$.

{\bf Subcase (i)} First consider $p=1$. From \eqref{needeq1}, we have
\begin{align*}
 \Re r_{f_{1}f_{2}}(-iy^{-1}) &= \Re r_{f_{1}}\left(-iy^{-1}\right) + \Re r_{f_{2}}\left(-iy^{-1}\right) +\Re (r_{f_{1}}\left(-iy^{-1}\right)r_{f_{2}}\left(-iy^{-1}\right)).
\end{align*}
Now from \eqref{ccor1.5b}, $$\frac{\Re r_{f_{2}}\left(-iy^{-1}\right)}{ \Re r_{f_{1}}\left(-iy^{-1}\right)} \to c \text{  as  } y \to \infty.$$
Also $$\frac{\Re r_{f_{1}}\left(-iy^{-1}\right)\Re r_{f_{2}}\left(-iy^{-1}\right)}{\Re r_{f_{1}}\left(-iy^{-1}\right)} = \Re r_{f_{2}}\left(-iy^{-1}\right) \to 0 \text{  as  } y \to \infty$$
and using \eqref{ccor1.5b} and \eqref{ccor1.5c} we observe that
\begin{align*}
\frac{\Im r_{f_{1}}\left(-iy^{-1}\right)\Im r_{f_{2}}\left(-iy^{-1}\right)}{\Re r_{f_{1}}\left(-iy^{-1}\right)} &=  \frac{\Im r_{f_{1}}\left(-iy^{-1}\right)}{y^{-(1-r/2)}}\frac{\Im r_{f_{2}}\left(-iy^{-1}\right)}{y^{-(1-r/2)}}\frac{y^{-(1+r/2)}}{\Re r_{f_{1}}\left(-iy^{-1}\right)}\frac{1}{y^{1-3r/2}} \to 0 \text{  as  }y \to \infty.
\end{align*}
Thus $ \Re r_{f_{1}f_{2}}(-iy^{-1}) \sim (1+c)\Re r_{f_{1}}\left(-iy^{-1}\right)$. Exactly same calculation taking the imaginary part gives us $ \Im r_{f_{1}f_{2}}(-iy^{-1}) \sim (1+c)\Im r_{f_{1}}\left(-iy^{-1}\right)$.
 Therefore we are done when $p=1$.

{\bf Subcase (ii)} Suppose $p \geqslant 2$. The real part can be dealt as in the proof of case \eqref{theoremod}.
For the imaginary part note that from \eqref{needeq2} we have  
 \begin{align*} 
\Im r_{f_{1}f_{2}}(-iy^{-1}) = \Im r_{f_{1}}\left(-iy^{-1}\right) + \Im r_{f_{2}}\left(-iy^{-1}\right) + O\left(y^{-1}\right).
\end{align*}
Now as $y \to \infty$, $y^{-(1-r/2)}\gg \Im r_{f_{1}}\left(-iy^{-1}\right)$, $\Im r_{f_{2}}\left(-iy^{-1}\right) \gg y^{-1}$ by \eqref{ccor1.5c}. Therefore we have $\Im r_{f_{1}f_{2}}(-iy^{-1}) \gg y^{-1}$. Noting the fact that $y^{-1} \ll y^{-(1-r/2)}$ finally we get $\Im r_{f_{1}f_{2}}(-iy^{-1}) \ll y^{-(1-r/2)}$, $r \in (0,1/2)$. \qed

\subsection*{Acknowledgements}
The first author's research is supported by fellowship from Indian Statistical Institute and the second author's research was supported by Cumulative Professional Development Allowance from Ministry of Human Resource Development, Government of India and Department of Science and Technology, Inspire funds. We also thank Octavio Arizmendi and the anonymous referee for helpful feedback.

\bibliographystyle{abbrvnat}
\bibliography{practicebib}
\end{document}